\documentclass{amsart}
\pdfoutput=1
\numberwithin{equation}{section}
\usepackage[scaled]{helvet} 
\usepackage{courier} 
\usepackage{graphicx}  
\usepackage{cite}
\usepackage{bbm}
\usepackage{comment}
\usepackage{verbatim}
\usepackage{mathrsfs}
\usepackage{mathtools}
\usepackage{latexsym}
\usepackage{epsfig}
\usepackage{amsmath}
\usepackage[usenames,dvipsnames]{xcolor}
\usepackage{bbm}
\usepackage{graphics}
\usepackage[utf8]{inputenc}
\usepackage{amssymb}
\usepackage{amsthm}
\usepackage{amsopn}
\usepackage{amscd}
\usepackage[all,knot]{xy}
\xyoption{all}
\usepackage{rotating}
\usepackage{tikz}
\usetikzlibrary{calc}
\usepgflibrary{shapes.geometric}
\usepgflibrary{shapes.misc}
\usetikzlibrary{positioning}
\usetikzlibrary{decorations}
\usetikzlibrary{decorations.pathreplacing}
\usepackage{tikz}
\usepackage{tikz-cd}
\mathtoolsset{showonlyrefs}

\usepackage[hidelinks]{hyperref}
\usepackage[alphabetic]{amsrefs}

\theoremstyle{plain}
\newtheorem{thm}{Theorem}[section]
\newtheorem{lem}[thm]{Lemma}
\newtheorem{prop}[thm]{Proposition}
\newtheorem{cor}[thm]{Corollary}

\newtheorem*{thm*}{Theorem}
\newtheorem*{lem*}{Lemma}
\newtheorem*{prop*}{Proposition}
\newtheorem*{cor*}{Corollary}

\theoremstyle{definition}
\newtheorem{defn}[thm]{Definition}
\newtheorem*{defn*}{Definition}
\newtheorem{construction}[thm]{Construction}

\newtheorem{ex}[thm]{Example}
\newtheorem{ex*}{Example}
\newtheorem{rem}[thm]{Remark}
\newtheorem*{rem*}{Remark}

\newtheorem{notation}[thm]{Notation}{}
\newtheorem{convention}[thm]{Convention}{}
{}
{}

\theoremstyle{remark}
{}
{}
{}

{}

\def\to{\longrightarrow} 
\renewcommand{\xrightarrow}[1]{\stackrel{#1}{\to}}

\def\Ab{\mathsf{Ab}}

\def\grMod{\mathsf{grMod}}

\def\EE{\mathbb{E}}

\def\NN{\mathbb{N}}

\def\ZZ{\mathbb{Z}}
\DeclareUnicodeCharacter{221E}{$\infty$}

\def\sfA{\mathsf{A}}

\def\sfC{\mathsf{C}}
\def\sfD{\mathsf{D}}

\def\sfK{\mathsf{K}}
\def\sfL{\mathsf{L}}

\def\sfR{\mathsf{R}}
\def\sfS{\mathsf{S}}
\def\sfT{\mathsf{T}}

\def\sfX{\mathsf{X}}

\def\mcA{\mathcal{A}}
\def\mcB{\mathcal{B}}

\def\mcD{\mathcal{D}}

\def\mcF{\mathcal{F}}
\def\mcG{\mathcal{G}}

\def\mcJ{\mathcal{J}}

\def\mcO{\mathcal{O}}
\def\mcP{\mathcal{P}}
\def\mcQ{\mathcal{Q}}

\def\mcS{\mathcal{S}}

\def\mcU{\mathcal{U}}
\def\mcV{\mathcal{V}}
\def\mcW{\mathcal{W}}

\def\mcX{\mathcal{X}}
\def\mcZ{\mathcal{Z}}

\def\bfB{\mathbf{B}}

\def\op{\mathrm{op}}

\def\perf{\sf{perf}}

\def\unit{\mathbf{1}}

\DeclareMathOperator{\add}{add}

\DeclareMathOperator{\colim}{colim}
\DeclareMathOperator{\hocolim}{hocolim}
\DeclareMathOperator{\Hom}{Hom}
\DeclareMathOperator{\coker}{coker}
\DeclareMathOperator{\End}{End}
\DeclareMathOperator{\id}{id}

\DeclareMathOperator{\Modu}{\mathsf{Mod}}

\DeclareMathOperator{\gr}{\mathsf{gr}}

\DeclareMathOperator{\Coh}{Coh}
\DeclareMathOperator{\QCoh}{QCoh}

\DeclareMathOperator{\Spec}{Spec}

\DeclareMathOperator{\cone}{cone}
\DeclareMathOperator{\Spc}{Spc}

\DeclareMathOperator{\supp}{supp}

\DeclareMathOperator{\loc}{\mathrm{loc}}
\DeclareMathOperator{\thick}{thick}

\DeclareMathOperator{\Vis}{Vis}

\DeclareMathOperator{\Shv}{Shv}

\definecolor{internationalkleinblue}{rgb}{0.0, 0.18, 0.65}

\definecolor{darkorange}{rgb}{1.00, 0.58, 0.00}

\newcommand\imCMsym[4][\mathord]{
	\DeclareFontFamily{U} {#2}{}
	\DeclareFontShape{U}{#2}{m}{n}{
		<-6> #25
		<6-7> #26
		<7-8> #27
		<8-9> #28
		<9-10> #29
		<10-12> #210
		<12-> #212}{}
	\DeclareSymbolFont{CM#2} {U} {#2}{m}{n}
	\DeclareMathSymbol{#4}{#1}{CM#2}{#3}
}
\newcommand\alsoimCMsym[4][\mathord]{\DeclareMathSymbol{#4}{#1}{CM#2}{#3}}

\imCMsym{cmmi}{124}{\CMjmath}
\imCMsym[\mathop]{cmsy}{113}{\CMamalg}
\imCMsym[\mathop]{cmex}{96}{\CMcoprod}
\alsoimCMsym[\mathop]{cmex}{97}{\CMbigcoprod}

\DeclareMathOperator{\qbasic}{\mathsf{QBasic}}
\DeclareMathOperator{\qbpresheaf}{\mathsf{QBPsh}}
\usepackage{pgfgantt}

\title{Associated Sheaf Functors in tt-geometry}

\author{James Rowe}
\address{James Rowe, School of Mathematics and Statistics,
	University of Glasgow,
	University Place,
	Glasgow G12 8QQ
}
\email{j.rowe.1@research.gla.ac.uk}
\begin{document}
	\begin{abstract}
		Given a tensor triangulated category we investigate the geometry of the Balmer spectrum as a locally ringed space. Specifically we construct functors assigning to every object in the category a corresponding sheaf and a notion of support based upon these sheaves. We compare this support to the usual support in tt-geometry and show that under reasonable conditions they agree on compact objects. We show that when tt-categories satisfy a scheme-like property then the sheaf associated to an object is quasi-coherent, and that in the presence of an appropriate t-structure and affine assumption, this sheaf is in fact the sheaf associated to the object's zeroth cohomology.  
	\end{abstract}
\maketitle
\let\clearpage\relax
\tableofcontents

\section{Introduction}
Given an essentially small tensor triangulated category, one can construct a topological space, called the Balmer spectrum, from the collection of prime tensor ideals. Important collections of subcategories such as localising or thick tensor ideals can often be classified by subsets of this space. Moreover, just as the spectrum of a commutative ring can be equipped with the usual structure sheaf, the Balmer spectrum also naturally admits the structure of a locally ringed space. In this work we investigate sheaves over the Balmer spectrum equipped with this structure sheaf, emphasising the geometry of the space and its influence on ideas of support and the cohomology of objects. 
More precisely, given a compactly generated tt-category $\sfT$ generated by the tensor unit $\unit$ we construct an \emph{associated sheaf functor} 
\[
[\unit, -]^\# : \sfT \to \Shv_{\mcO_\sfT^\bullet}(\Spc(\sfT^c))
\]

Using the geometric information of this sheaf and its stalks we can define a support for an object $X \in \sfT$ by setting $\supp^\bullet(X,\unit)=\{\mcP \ \vert \ [\unit, X]^\#_{\mcP}\neq 0\}$. This notion of support satisfies many of the desirable properties usually satisfied by support theories and it is natural to ask how this geometric flavour of support compares to the usual notions. We have the following comparison:

\begin{thm*}
	Suppose $\sfT^c=\thick(\unit)$ and the Balmer spectrum of $\sfT^c$ is a noetherian topological space.
	Then for all objects $X \in \sfT$ there is containment
	\[\supp X \subseteq \supp^\bullet(X,\unit),\]
	where the support on the left is the support in the sense of Balmer-Favi.
	Moreover if $X$ is compact then there is an equality
	\[\supp X = \supp^\bullet(X,\unit).\]
	
\end{thm*}

In algebraic geometry, the notion of affine spaces and schemes are both powerful and prolific. By using the spectrum as an avatar for our tt-category we can pull these geometric properties into our setting. Specifically we consider Balmer's comparison map of locally ringed spaces $\rho : (\Spc(\sfT^c),\mcO_\sfT)\to (\Spec(R_\sfT),\mcO_{R_\sfT})$ between the Balmer spectrum and the spectrum of the endomorphism ring of the tensor unit. We call categories where $\rho$ is an isomorphism \emph{affine}, and call a category \emph{schematic} if it is locally affine with respect to some open cover of the spectrum. With this we obtain the following theorem:
\begin{thm*}
	Let $\sfT$ be schematic. Then for every $X \in \sfT$, the sheaf $[\unit, X]^\#$ is quasi-coherent on $(\Spc(\sfT^c),\mcO_\sfT).$
\end{thm*}

If a triangulated category admits a generator satisfying certain conditions, then the generator equips the category with a natural t-structure. This allows one to ask questions about the cohomology of objects. Our construction of the support $\supp^\bullet(X, \unit)$ includes information about the suspensions of the object $X$. We can construct a slightly different support $\supp(X, \unit)$ without this suspension data, which we call the untwisted support.    
Combining these ideas with the machinery of t-structures allows us to obtain the following theorem:
\begin{thm*}
	Let $\sfT$ be an affine category generated by the tensor unit $\unit$. Assume that $\Hom_\sfT(\unit, \Sigma^n \unit)=0$ for all $n>0$. Then for all objects $X \in \sfT$ we have
	\[\supp^\bullet (X, \unit) = \bigcup_{i \in \ZZ}\supp(H^i(X), \unit).\]
\end{thm*}

We will now lay out the format of the remainder of this article. In \S2 we give the general preliminaries on tt-categories and actions of tt-categories. In \S3 we give the details of the Balmer spectrum, including its formulation as a locally ringed space. We also recall the definition of support theories obtained from tensor triangulated categories acting on triangulated categories. \S4 contains our construction of relative sheaf functors, in the generality of a triangulated category acted on by a tt-category, and including twisting by an invertible object. In \S5 we compare the support produced by these sheaf functors to the usual notion of support. \S6 details the conditions under which properties such as (quasi-)coherence can lead to the formation of thick subcategories. Affine and schematic categories are analysed in \S7, leading us to show that in such categories the sheaves associated to objects are always quasi-coherent. We also introduce the notion of a quasi-affine tt-category and investigate the associated comparison maps. \S8 compares the usual gluing of sheaves with the tensor-triangular notion of gluing over a Mayer-Vietoris cover and the significant differences between the operations. Finally in \S9 we show how the sheaf functors interact with naturally occurring t-structures and work towards the theorem showing that when our category is affine, the twisted support of an object can be obtained by considering the untwisted supports of its cohomologies.

\section{Preliminaries}

In this section we lay out the framework of triangulated categories and the action of a tensor-triangulated category. Details on the axioms of triangulated categories can be found in \cite{NeeCat}.

\begin{defn}\cite{BaSpec}
	A \emph{tensor triangulated category} is a triple $(\sfT, \otimes, \unit)$ consisting of a triangulated category $\sfT$, a symmetric monoidal product $\otimes: \sfT\times \sfT \to \sfT$ which is exact in each variable, with unit $\unit$.
\end{defn}

Given a triangulated category there are subcategories preserving various parts of the triangulated structure.

\begin{defn}
	Given a triangulated category $\sfT$, a subcategory is said to be \emph{thick} if it is triangulated and closed under direct summands. Explicitly, a full subcategory is thick if it is closed under suspensions, cones, and direct summands. Given a collection of objects $X$ in $\sfT$, we denote by $\thick(X)$ the smallest thick subcategory of $\sfT$ containing $X$. It is referred to as the \emph{thick subcategory generated by $X$.}
\end{defn}

\begin{defn}
	A subcategory of $\sfT$ is said to be \emph{localising} if it is triangulated and closed under arbitrary coproducts. Given a collection of objects $X$ in $\sfT$, we denote by $\loc(X)$ the smallest localising subcategory of $\sfT$ containing $X$. It is referred to as the \emph{localising subcategory generated by $X$.}
\end{defn}

\begin{lem}
	Every localising subcategory is thick.
\end{lem}

\begin{proof}
	Consider a localising subcategory $\mcS \subseteq \sfT$ containing an object of the form $Z=X\oplus Y.$ As $\mcS$ is localising it is closed under arbitrary coproducts and so contains the coproduct $\coprod_{n \in \NN}(X\oplus Y)\cong X\oplus \coprod_{n \in \NN}(Y\oplus X)$. Therefore $\mcS$ contains the triangle
	\[X\to \coprod_{n \in \NN}(X\oplus Y) \to \coprod_{n \in \NN}(X\oplus Y)\to \Sigma X\]
	and so $\mcS$ contains $X$. Therefore $\mcS$ is thick.
\end{proof}

While we allow tensor triangulated categories to be large, some smallness conditions are needed in order to obtain good behaviour. These conditions are satisfied by many examples of interest.

\begin{defn}
	Let $\sfT$ be a triangulated category admitting all set-indexed coproducts. An object $t \in \sfT$ is \emph{compact} if $\Hom_{\sfT}(t,-)$ preserves arbitrary coproducts. That is, for any family $\{x_{\lambda} \ \vert \ \lambda \in \Lambda\}$ we have
	$\Hom_{\sfT}(t,\CMcoprod_{\lambda \in \Lambda}x_\lambda) \cong \bigoplus_{\lambda \in \Lambda} \Hom_{\sfT}(t,x_{\lambda}).$
	We say $\sfT$ is \emph{compactly generated} if there is a set of compact objects $\mcG$ such that an object $t \in \sfT$ is zero if and only if $\Hom_{\sfT}(g,\Sigma^i t)=0$ for all $g \in \mcG$ and all $i \in \mathbb{Z}$. 
	We denote by $\sfT^c$ the full subcategory of compact objects in $\sfT.$ Note that $\sfT^c$ is an essentially small thick subcategory of $\sfT$.
\end{defn}

\begin{defn}
	Assume that $\sfT$ is closed symmetric monoidal. Then there is an associated internal hom functor denoted $\hom(-,-)$. The \emph{dual} of an object $k$ is given by $k^{\vee}:=\hom(k,\unit)$. An object $k$  is \emph{rigid} if for all other objects $t$ we have that the natural evaluation map $k^{\vee}\otimes t \to \hom(k,t)$ is an isomorphism. For a rigid object $k$ we have $k\cong (k^{\vee})^{\vee}$.
\end{defn}

We can now combine the structure of a tensor triangulated category with the smallness conditions of compactness and rigidity to obtain a good structure to analyse.

\begin{defn}
	A \emph{rigidly-compactly generated tensor triangulated category} is a triple $(\sfT,\otimes,\unit)$ where $\sfT$ is a compactly generated tensor triangulated category, and $(\otimes,\unit)$ is a symmetric monoidal structure on $\sfT$ such that the tensor product $\otimes$ is a coproduct preserving exact functor in each variable, and the compact objects $\sfT^c$ form a rigid tensor subcategory. In particular we require $\unit$ to be compact. We will refer to such a category $\sfT$ as a \emph{big tt-category}.
\end{defn}

From now on we will suppress notation and take $\sfT$ to be a big tt-category and denote by $\sfT^c$ the full subcategory of compact objects.

\begin{defn}
	An object $u$ is said to be \emph{invertible} if there exists an object $v$ such that $u\otimes v\cong \unit$.
\end{defn} 

\begin{defn}
	\label{defgradedhom}
	For any two objects $a,b$ in a big tt-category $\sfT$, given an invertible object $u$ we define the \emph{twisted homomorphism group}
	\[\Hom^{\bullet}_{\sfT}(a,b):=\bigoplus_{i\in\ZZ}\Hom_{\sfT}(a,u^{\otimes i} \otimes b).\]
	In the special case \(u=\Sigma\unit\) we define the \emph{graded homomorphism group} by
	\[\Hom^{\ast}_\sfT(a,b):=\bigoplus_{i\in\ZZ}\Hom_{\sfT}(a,\Sigma^i b).\]
\end{defn}

\begin{rem}
	For completeness we should include the invertible object $u$ in the notation, but we suppress it here. We rarely deal with more than one invertible object at a time in this setup. Also note that in the case $u=\unit$, the twisted homomorphism group is given by \[\Hom^{\bullet}_{\sfT}(a,b)=\bigoplus_{i\in\ZZ}\Hom_{\sfT}(a,\unit^{\otimes i} \otimes b)=\bigoplus_{i\in\ZZ}\Hom_{\sfT}(a,b),\]
	which is a coproduct of countably many copies of the usual homomorphism group.  
\end{rem}

There is a useful connection between the twisted homomorphism group and the generators of a category.

\begin{lem}
	\label{lemthickzero}
	If $\sfA=\thick(g)$ is a triangulated category and for all $j\in\ZZ$ we have $\Hom_{\sfA}(g,\Sigma^j x)=0$ then $x=0.$
\end{lem}

\begin{proof}
	Consider the collection 
	\[\prescript{\perp}{}{x}=\{y\in\sfA\ \vert \ \forall j \in \ZZ:\Hom_{\sfA}(y,\Sigma^j x)=0\}.\] 
	Observe that \(\prescript{\perp}{}{x}\) is thick. 
	By assumption \(g\in \prescript{\perp}{}{x}\) and so \(\prescript{\perp}{}{x}=\sfA.\) 
	Therefore \(\Hom_\sfA(x,x)=0.\)
\end{proof}

\begin{rem}
	The above lemma holds when $\sfA=\loc(g)$ for a compact object $g$.
\end{rem}

We have seen the definition of thick and localising subcategories, but the definitions given have no mention of any tensor product. We capture the monoidal structure by considering the \emph{ideals} of our big tt-category. 

\begin{defn}\cite[Def 1.2]{BaSpec}
	A \emph{thick tensor-ideal} $\sfA$ of $\sfT$ is a thick subcategory such that for all $a \in \sfA$ and $t \in \sfT$ the tensor product $a\otimes t$ also belongs to $\sfA$. Just as with thick subcategories, we will denote by $\thick^\otimes(X)$ the smallest thick tensor-ideal containing $X$.
\end{defn}

Following the similarity with ideals in the usual theory of commutative rings we will also be interested in those ideals which have the additional property of being \emph{prime}.

\begin{defn}\cite[Def 2.1]{BaSpec}
	A \emph{prime ideal} of $\sfT$ is a proper thick tensor-ideal $\mcP \subsetneq \sfT$ such that if $a\otimes b \in \mcP$ then $a \in \mcP$ or $b \in \mcP.$
\end{defn}

The definitions so far suggest that one can think of a big tt-category as some kind of bizarre commutative ring. The work of Stevenson in \cite{StevensonActions} takes this comparison further by introducing the action of a big tt-category on another triangulated category. In the comparison with rings, this is the introduction of modules.

\begin{rem}
	From now on we will assume that the suspension functor of the big tt-category $\sfT$ is compatible with the tensor product. This is guaranteed when $\sfT$ can be realised as the homotopy category of a finitely presentable stable $\EE_2$-monoidal $\infty$-category.
\end{rem}

\begin{defn}\cite[Def 3.2]{StevensonActions}
	Let $\sfT$ be a big tt-category and $\sfK$ a triangulated category. A \emph{left action} of $\sfT$ on $\sfK$ is a functor 
	\[\ast:\sfT \times \sfK \to \sfK\]
	which is exact in each variable, together with natural isomorphisms 
	\[a_{X,Y,A}:(X\otimes Y)\ast A \xrightarrow{\sim}X\ast(Y\ast A)\]
	and
	\[l_A:\unit \ast A \xrightarrow{\sim}A\]
	for all $X,Y \in \sfT$, $A \in \sfK$, compatible with the biexactness of $(-)\ast(-)$ and satisfying the following conditions:
	\begin{enumerate}
		\item The associator $a$ satisfies the pentagon condition which asserts that the following diagram commutes for all $X,Y,Z\in\sfT$ and $A \in \sfK$ 
		\[
		\begin{tikzcd}
			& X\ast(Y\ast(Z\ast A)) &                                  \\
			X\ast((Y\otimes Z)\ast A) \arrow[ru, "X\ast a_{Y,Z,A}"] &   & (X\otimes Y)\ast(Z\ast A) \arrow[lu, "a_{X,Y,Z\ast A}"']                \\
			(X\otimes(Y\otimes Z))\ast A \arrow[u, "a_{X,Y\otimes Z,A}"]  &   & ((X\otimes Y)\otimes Z)\ast A \arrow[ll, ""] \arrow[u, "a_{X\otimes Y,Z,A}"']
		\end{tikzcd}\]
		where the bottom arrow is the associator of $\sfT$.
		\item The unitor $l$ makes the following squares commute for every $X \in \sfT$ and $A \in \sfK$
		\[
		\begin{tikzcd}
			X\ast(\unit \ast A) \arrow[r, "X\ast l_A"]           & X\ast A \arrow[d, "1_{X\ast A}"] & \unit\ast(X\ast A) \arrow[r, "l_{X\ast A}"]           & X\ast A \arrow[d, "1_{X\ast A}"] \\
			(X\otimes \unit)\ast A \arrow[u, "a_{X,\unit,A}"] \arrow[r] & X\ast A                & (\unit \otimes X)\ast A \arrow[r] \arrow[u, "a_{\unit,X,A}"] & X\ast A               
		\end{tikzcd}\]
		where the bottom arrows are the right and left unitors of $\sfT$.
		\item For every $A \in \sfK$ and $r,s \in\ZZ$ the diagram
		\[
		\begin{tikzcd}
			\Sigma^r \unit \ast \Sigma^s A \arrow[r, "\sim"] \arrow[d, "\sim"] & \Sigma^{r+s}A \arrow[d, "(-1)^{rs}"] \\
			\Sigma^r(\unit \ast \Sigma^s A) \arrow[r, "\sim"]                & \Sigma^{r+s}A               
		\end{tikzcd}\]
		is commutative, where the left vertical map comes from exactness in the first varianble of the action, the bottom horzontal map is the unitor, and the top map is given by the composite
		\[\Sigma^r\unit \ast \Sigma^s A \to \Sigma^s(\Sigma^r\unit \ast A)\to \Sigma^{r+s}(\unit \ast A) \xrightarrow{l}\Sigma^{r+s}A\]
		whose first two maps use exactness in both variables of the action. 
		\item The functor $\ast$ distributes over coproducts whenever they exist. That is, for families $\{X_i\}_{i\in I}$ in $\sfT$ and $\{A_j\}_{j\in J}$ in $\sfK$, and $X$ in $\sfT$, $A$ in $\sfK$ the canonical maps 
		\[\CMcoprod_i (X_i \ast A)\xrightarrow{\sim} (\CMcoprod_i X_i) \ast A\]
		and
		\[\CMcoprod_j (X\ast A_j)\xrightarrow{\sim} X\ast (\CMcoprod_j A_j)\]
		are isomorphisms whenever the coproducts concerned exist. 
	\end{enumerate}
	
	If $\sfT$ acts on $\sfK$ we will say that $\sfK$ is a $\sfT$-module. 
\end{defn}

\begin{defn}
	Let $\sfL\subseteq\sfK$ be a localising (thick) subcategory. We say $\sfL$ is a localising $\sfT$\emph{-submodule} of $\sfK$ if the functor \[\sfT\times \sfL \xrightarrow{\ast}\sfK\] factors via $\sfL$. That is, $\sfL$ is closed under the action of $\sfT$. Given a collection of objects $\mcA$ in $\sfK$, we denote by $\loc^\ast(\mcA)$ (resp. $\thick^\ast(\mcA)$) the smallest localising (resp. thick) submodule containing $\mcA$.
\end{defn}

We have the following useful lemma.

\begin{lem}\cite[3.13]{StevensonActions}
	If $\sfT$ is generated as a localising subcategory by the tensor unit $\unit$, then every localising subcategory of $\sfK$ is a $\sfT$-submodule.
\end{lem}

\section{The Balmer Spectrum}
We now present the Balmer spectrum of a tensor triangulated category and the first notion  of support as given in \cite{BaSpec}. Restricting to the compact objects, this will let us view the collection of prime ideals of a tensor triangulated category as a topological space, equipped with a universal support theory satisfying many desirable properties.

\begin{defn}
	Let $\sfT$ be a big tt-category, with compacts $\sfT^c$. The \emph{Balmer spectrum} of $\sfT^c$ is given by
	\[\Spc(\sfT^c)=\{\mcP \ \vert \ \mcP\text{ prime ideal of }\sfT^c\}.\]
	For all $a \in \sfT^c$ we define the open subsets $U(a)=\{\mcP \in \Spc(\sfT^c)\ \vert \ a \in \mcP\}.$ This forms a basis for the topology on $\Spc(\sfT^c)$.
\end{defn}

\begin{defn}
	For an object $t \in \sfT^c$ the (small) tt-support $\supp_{\sfT^c}t$ is defined as
	\[\supp_{\sfT^c}t = \{\mcP \in \Spc(\sfT^c)\ \vert \ t \not\in \mcP\}.\]
	By definition $\supp_{\sfT^c}t$ is the complement of the basic open subset $U(t)$.
	Given a subset of objects \(\mcJ\subset \sfT^c\) we define the support of the subset as
	\[\supp_{\sfT^c}(\mcJ):=\bigcup_{j \in \mcJ}\supp_{\sfT^c}(j).\] 
\end{defn}

This notion of support has many desirable properties and is in fact universal amongst such constructions. 

\begin{thm}\emph{Universal property of the spectrum \cite[3.2]{BaSpec}}
	We have
	\begin{enumerate}
		\item $\supp_{\sfT^c}(0)=\varnothing$ and $\supp_{\sfT^c}(\unit) = \Spc(\sfT^c)$.
		\item $\supp_{\sfT^c}(a\oplus b)=\supp_{\sfT^c}(a) \cup \supp_{\sfT^c}(b)$. 
		\item $\supp_{\sfT^c}(\Sigma a)=\supp_{\sfT^c}(a)$ where $\Sigma$ is the suspension functor for $\sfT$. 
		\item $\supp_{\sfT^c}(a) \subseteq \supp_{\sfT^c}(b)\cup\supp_{\sfT^c}(c)$ for any exact triangle $a\to b\to c \to \Sigma a$.
		\item $\supp_{\sfT^c}(a\otimes b)=\supp_{\sfT^c}(a)\cap \supp_{\sfT^c}(b)$.\\
		Moreover, for any pair $(X,\sigma)$, where $X$ is a topological space and $\sigma$ an assignment of closed subsets $\sigma(t) \subseteq X$ to objects $t \in \sfT^c$ which satisfy properties (a)-(e) above, there exists a unique continuous map $f: X \to \Spc(\sfT^c)$ such that $\sigma(t)=f^{-1}(\supp_{\sfT^c}(t))$.
	\end{enumerate}
	
\end{thm}

As noted, the issue with this notion of support is that it is only defined for compact objects, and one cannot simply extend the definition of the spectrum to include non-compact objects without losing the useful properties of the spectrum. A solution to this problem was defined by Balmer-Favi in \cite{BaRickard}, which was then extended by Stevenson in \cite{StevensonActions} to the case of actions. As the techniques used in the definition will be used in various future sections, we will present them in some detail. We focus on the three main components:

\begin{itemize}
	\item Localisations with respect to the topology on $\Spc(\sfT^c).$
	\item Tensor idempotents associated to points in the spectrum.
	\item Support defined in terms of tensor idempotents.
\end{itemize}

We begin with the notion of a localisation sequence.

\begin{defn}
	A \emph{localisation sequence} is a diagram
	\[
	\begin{tikzcd}
		\sfR \arrow[rr, "i_\ast", shift left=3] \arrow[phantom,rr, "\perp" description] &  & \sfT \arrow[ll, "i^!", shift left=3] \arrow[rr, "j^\ast", shift left=3] \arrow[phantom,rr, "\perp" description] &  & \sfS \arrow[ll, "j_\ast", shift left=3]
	\end{tikzcd}
	\]
	where both \(i_\ast\) and \(j_\ast\) are fully faithful, and we have equalities \((i_\ast \sfR)^\perp = j_\ast \sfS\) and \(\prescript{\perp}{}{(j_\ast \sfS)}=i_\ast \sfR\) where
	\[(i_\ast \sfR)^\perp = \{t \in \sfT \ \vert \ \Hom_\sfT (i_\ast r, t)=0 \text{ for all }r \in \sfR\}\]
	and
	\[\prescript{\perp}{}{(j_\ast \sfS)}=\{t \in \sfT \ \vert \ \Hom_\sfT(t,j_\ast s)=0 \text{ for all } s \in \sfS\}.\]
\end{defn}

\begin{lem}\cite[5.5.1]{KrLoc}
	Given a localisation sequence as above, the functor $i^!$ is coproduct preserving if and only if $j_\ast$ is coproduct preserving.
\end{lem}

\begin{defn}
	A localisation sequence is called \emph{smashing} if $i^!$ (or equivalently $j_\ast$) preserves coproducts. In such a sequence we call $\sfR$ a \emph{smashing subcategory} of $\sfT$. If $\sfR$ is a tensor ideal then we say it is a \emph{smashing tensor-ideal}.
\end{defn}

In particular we will use the following theorem, which is the work of Miller \cite{Miller} (for part 1) and Neeman \cite{NeeCat} (for parts 2 and 3). This particular form of the theorem is from \cite[4.1]{BaRickard}   

\begin{thm}\emph{[Miller, Neeman]}
	\label{millerneeman}
	Let $\sfT$ be a big tt-category and let $\sfC$ be a thick \(\otimes\)-ideal of $\sfT^c$. Then we have
	
	\[
	\begin{tikzcd}
		\sfC \arrow[rr, hook] \arrow[d, hook]                              &  & \sfT^c \arrow[rr] \arrow[d, hook]                                                       &  & \sfT^c/\sfC \arrow[d, hook]               \\
		\loc(\sfC) \arrow[rr, shift left=3] \arrow[phantom,rr, "\perp" description] &  & \sfT \arrow[rr, shift left=3] \arrow[ll, shift left=3] \arrow[phantom,rr, "\bot" description] &  & \sfT/\loc(\sfC) \arrow[ll, shift left=3]
	\end{tikzcd}
	\]
	
	\begin{enumerate}
		\item $\loc(\sfC)$ is a smashing tensor ideal, the bottom row of the diagram is a smashing localisation sequence and $\loc(\sfC)^c = \loc(\sfC)\cap \sfT^c = \sfC.$
		\item $\sfT/\loc(\sfC)$ has small hom-sets and is a compactly generated tensor-triangulated category.
		\item $\sfT^c/\sfC$ fully faithfully embeds into the compact objects of $\sfT/\loc(\sfC)$ and the additive closure of $\sfT^c/\sfC$  is exactly $(\sfT/\loc(\sfC))^c.$ That is, if $t$ is a compact object in $\sfT/\loc(\sfC)$ then $t$ is a summand of an object in $\sfT^c/\sfC$. 
	\end{enumerate}
\end{thm}

We can choose $\sfC$ in the theorem to connect it to the topology on $\Spc(\sfT^c)$ via Thomason subsets.

\begin{defn}
	A subset $\mcV \subset \Spc(\sfT^c)$ is \emph{Thomason} if $\mcV$ is a (possibly infinite) union of closed subsets with quasi-compact open complements. 
\end{defn}

Given such a Thomason subset $\mcV$ define the associated thick-tensor ideal $\tau(\mcV)=\{t \in \Spc(\sfT^c)\ \vert \ \supp_{\sfT^c}t \subseteq \mcV\}.$ Letting $\sfC=\tau(\mcV)$ and $\Gamma_\mcV\sfT=\loc(\sfC)$ the above theorem gives us a smashing localisation sequence   

\[
\begin{tikzcd}
	\Gamma_\mcV\sfT \arrow[rr, "", shift left=3] \arrow[phantom,rr, "\perp" description] &  & \sfT \arrow[ll, "", shift left=3] \arrow[rr, "", shift left=3] \arrow[phantom,rr, "\perp" description] &  & L_\mcV\sfT \arrow[ll, "", shift left=3]
\end{tikzcd}
\]

with corresponding acyclisation and localisation functors $\Gamma_\mcV$ and $L_\mcV$ respectively.

With such a localisation we can now move onto the definition of tensor idempotents associated to points. 

\begin{defn}
	We say a point $x \in\Spc(\sfT^c)$ is \emph{visible} if there exist Thomason subsets $\mcV$ and $\mcW$ of $\Spc \sfT^c$ such that 
	\[\mcV\setminus (\mcV \cap \mcW)=\{ x \}.\]
	We denote the collection of all such points by $\Vis(\sfT^c).$ 
	Given such Thomason subsets we define a tensor idempotent 
	\[\Gamma_x \unit = \Gamma_\mcV \unit \otimes L_\mcW \unit.
	\]
	By \cite[Corollary 7.5]{BaRickard} any such pairs $(\mcV_1, \mcW)$ and $(\mcV_2, \mcW_2)$ of Thomason subsets will define isomorphic tensor idempotents.
	We define the following particular case:
	For each point $x \in\Spc(\sfT^c)$ define subsets of the spectrum $\mcV(x)=\overline{\{x\}}$ and $\mcZ(x)=\{y \in \Spc(\sfT^c)\ \vert x \not\in \mcV(y)\}.$ If $\mcV(x)$ is a Thomason subset then we define a tensor idempotent
	\[\Gamma_x \unit = \Gamma_{\mcV(x)}\unit \otimes L_{\mcZ(x)}\unit.\]
\end{defn}

We can now introduce the definition of the support relative to an action.

\begin{defn}
	Let $\sfT$ act on $\sfK$. The for $A \in \sfK$ we define the support of $A$ to be the set
	\[\supp_{(\sfT,\ast)}A=\{x \in \Vis(\sfT^c)\ \vert \ \Gamma_x A\neq 0\}.\]
	When the action in question is clear we will omit the subscript from the notation.
\end{defn}

The below properties demonstrate the good behaviour of this support:

\begin{prop}\cite[5.7]{StevensonActions}
	The support assignment $\supp_{(\sfT,\ast)}$ satisfies the following properties:
	\begin{itemize}
		\item given a triangle
		\[A\to B\to C \to \Sigma A\] 
		in $\sfK$ we have $\supp B \subseteq \supp A \cup \supp C;$
		\item for any $A \in \sfK$ and $i \in \ZZ$
		\[\supp A =\supp \Sigma^i A;\]
		\item given a set-indexed family $\{A_\lambda\}_{\lambda \in \Lambda}$ of objects of $\sfK$ there is an equality 
		\[\supp \CMcoprod_\lambda A_\lambda = \bigcup_\lambda \supp A_\lambda\]
		whenever the coproduct on the left exists;
		\item the support satisfies the separation axiom. That is, for every specialisation closed subset $\mcV\subseteq \Vis \sfT^c$ and every object $A$ of $\sfK$ 
		\begin{align*}
			\supp \Gamma_\mcV \unit \ast A &= (\supp A)\cap \mcV\\
			\supp L_\mcV \unit \ast A &= (\supp A)\cap (\Vis \sfT^c \setminus \mcV). 
		\end{align*}
	\end{itemize}
\end{prop}

Considering a big tt-category acting on itself via the tensor product, the following additional properties hold.

\begin{prop}\cite[7.17 (a)]{BaRickard}
	Let $\sfT$ be a tt-category acting on itself via the tensor product. Suppose in addition that the spectrum $\Spc(\sfT^c)$ is a noetherian topological space. Then 
	\begin{itemize}
		\item for every compact object $t \in \sfT^c$ we have
		\[\supp_{(\sfT,\ast)}t = \supp_{\sfT^c}t.\]
		\item $\supp_{(\sfT,\ast)}(0)=\varnothing$ and $\supp_{(\sfT,\ast)}(\unit) = \Spc(\sfT^c).$
		\item $\supp_{(\sfT,\ast)}(t\otimes t^\prime) \subseteq \supp_{(\sfT,\ast)}(t) \cap \supp_{(\sfT,\ast)}(t^\prime).$
	\end{itemize}
\end{prop}

From now on we will simply write $\supp$ for any of the support theories detailed so far, unless it is useful to specify further.	

In \cite{BaSSS} Balmer constructs a locally ringed space structure on the spectrum $\Spc(\sfT^c)$. We will study certain collections of sheaves of modules on this space. In order to do so we will provide Balmer's construction and adjusted notation.

Throughout we assume $\sfT$ is a big tt-category with compact objects $\sfT^c.$

\begin{defn}\cite[2.1]{BaSSS}
	We define the \emph{central ring} $R_{\sfT}$ to be the endomorphism ring of the tensor unit
	\[R_{\sfT}:=\End_{\sfT}(\unit).\]
	Given a compact invertible object $u$ we define the \emph{twisted central ring} $R^{\bullet}_{\sfT}$ to be 
	\[R^{\bullet}_{\sfT} := \Hom^{\bullet}_{\sfT}(\unit,\unit).\]
	Note that this is a special case of Definition \ref{defgradedhom} where $a=b=\unit.$
\end{defn}

\begin{prop}\cite[3.3]{BaSSS}
	Let $u \in \sfT^c$ be invertible. Then there exists $\epsilon_u \in R_{\sfT^c}$ such that $R^{\bullet}_{\sfT}$ is $\epsilon_u$-commutative. That is, given two homogenous elements $f$ and $g$ of orders $i$ and $j$ respectively, we have $fg=\epsilon_u^{ij}gf.$ 
\end{prop}

We note that the space $\Spc(\sfT^c)$ has a basis of quasi-compact open subsets \cite[2.7, 2.14]{BaSpec}.

\begin{construction}\cite[6.1]{BaSSS}
	Let $U \subset \Spc(\sfT^c)$ be a quasi-compact open subset with closed complement $Z$. Define $\sfT^c_Z :=\{t \in \sfT^c \ \vert \ \supp(t)\subset Z\}$ to be the thick $\otimes$-ideal of $\mathcal{T}^c$ supported outside of $U$. Define the tensor-triangulated category $\sfT^c$ on $U$
	\[\sfT^c(U):=(\sfT^c/\sfT^c_Z)^\natural\]
	as the idempotent completion of the Verdier quotient $\sfT^c/\sfT^c_Z.$
	This quotient is the localisation $S^{-1}\sfT^c$ with respect to $S=\{s:a\to b \ \vert \ \supp(\cone(s)) \subset Z\}.$
	
	\begin{rem}
		Note that for every quasi-compact open $U\subset \Spc(\sfT^c)$ we have $U\cong \Spc(\sfT^c(U))$ and moreover if $V$ is a quasi-compact subset of $U$ then $(\sfT^c(U))(V)\cong \sfT^c(V).$ 
	\end{rem}
	
	By construction we have a natural monoidal functor 
	\[q_U: \sfT^c \to \sfT^c(U).\]

\end{construction}

\begin{notation}
	Given a (partially defined) presheaf $F$ we denote the associated sheaf by $F^\#$.
\end{notation}

\begin{defn}\cite[6.4]{BaSSS}
	For each quasi-compact open $U \subset \Spc(\sfT^c)$, define the commutative ring $\prescript{}{\mathrm{p}}{\mcO}_{\sfT}(U)$ by
	\[\prescript{}{\mathrm{p}}{\mcO}_{\sfT}(U):=R_{\sfT(U)}=\Hom_{\sfT(U)}(\unit_U, \unit_U).\]
	For an invertible object $u$ define the $\epsilon_u$-commutative graded ring 
	\[\prescript{}{\mathrm{p}}{\mcO}^{\bullet}_{\sfT}(U):=R^{\bullet}_{\sfT(U)}=\Hom^{\bullet}_{\sfT(U)}(\unit_U, \unit_U).\]
	These form partially defined presheaves on $\Spc(\sfT^c)$, only defined on the basis of quasi-compact open sets. The associated sheaves on the space $\Spc(\sfT^c)$ are denoted
	\begin{align*}
		\mcO_{\sfT}&:=\prescript{}{\mathrm{p}}{\mcO}_{\sfT}^\#,\\
		\mcO^{\bullet}_{\sfT}&:=(\prescript{}{\mathrm{p}}{\mcO}^{\bullet}_{\sfT})^\#.
	\end{align*}
	
	We denote the locally ringed space obtained by
	\[
	\Spec(\sfT)=(\Spc(\sfT^c),\mcO_{\sfT}),\]
	and the graded locally ringed space by
	\[\Spec^{\bullet}(\sfT)=(\Spc(\sfT^c),\mcO^{\bullet}_{\sfT}).\]
	
\end{defn}

There is a natural map of locally ringed spaces between the Balmer spectrum of $\sfT^c$ and the spectrum of the commutative ring $R_{\sfT}.$

\begin{construction}\cite[5.6, 6.10]{BaSSS}
	Let $\mcP \in \Spc(\sfT^c)$ and define 
	\[\rho^{\bullet}_{\sfT}(\mcP):=\{f \in (R^{\bullet}_{\sfT})^{\hom}\ \vert \ \cone(f)\not\in\mcP\}.\]
	By \cite[5.6]{BaSSS} $\rho^{\bullet}_{\sfT}(\mcP)$ is a homogenous prime ideal of $R^{\bullet}_{\sfT}$. Moreover, $\rho^{\bullet}_{\sfT}:\Spc(\sfT^c)\to \Spec^{\mathrm{h}}(R^{\bullet}_{\sfT})$ is continuous and natural in $\sfT^c$.
	For every $s \in (R^{\bullet}_{\sfT})^\mathrm{even}$ we define \[U(s):=U(\cone(s))=\{\mcP\in\Spc(\sfT^c)\ \vert \ \cone(s) \in \mcP\}.\] For each distinguished open $D(s)$ in $\Spec(R_{\sfT})$ we have $U(s) = (\rho^{\bullet}_{\sfT})^{-1}(D(s)).$ 
	By \cite[Lemma 6.9]{BaSSS} we have
	\[\mcO^{\bullet}_{\Spec^{\mathrm{h}}(R^\bullet_\sfT)}(D(s))\cong \prescript{}{\mathrm{p}}{\mcO}^{\bullet}_{\sfT}(U(s))\]
	and both are naturally isomorphic to $R^{\bullet}_{\sfT}[s^{-1}].$
	This allows us to construct a ring homomorphism
	\[r_{D(s)}: \mcO^{\bullet}_{\Spec^{\mathrm{h}}(R^{\bullet}_{\sfT})}(D(s))\to \mcO^{\bullet}_{\sfT}(U(s))\]
	as the composition of the isomorphism $\mcO^{\bullet}_{\Spec^{\mathrm{h}}(R^\bullet_\sfT)}(D(s))\cong \prescript{}{\mathrm{p}}{\mcO}^{\bullet}_{\sfT}(U(s))$ followed by the sheafification morphism.
	This construction is compatible with restriction and defines a morphism of ringed spaces
	\[(\rho^{\bullet}_{\sfT},r):\Spec^\bullet(\sfT) \to \Spec^{\mathrm{h}}(R^{\bullet}_{\sfT}).\]
\end{construction}

\begin{prop}\emph{\cite[6.11]{BaSSS}}
	The map 
	\[(\rho^{\bullet}_{\sfT},r):\Spec^\bullet(\sfT) \to \Spec^{\mathrm{h}}(R^{\bullet}_{\sfT})\]
	is a map of locally ringed spaces. Moreover, if $\rho^{\bullet}_{\sfT}$ is a homeomorphism, then $(\rho^{\bullet}_{\sfT},r)$ is an isomorphism of graded locally ringed spaces.
	The degree $0$ restriction of this map defines a map of locally ringed spaces  
	\[(\rho_{\sfT},r):\Spec(\sfT) \to \Spec(R_{\sfT})\] such that if $\rho_\sfT$ is a homeomorphism then $(\rho_\sfT,r)$ is an isomorphism of locally ringed spaces.
\end{prop}

The proofs for the above require the following useful lemma. We will use more general versions at various points in this work.

\begin{lem}\cite[6.3]{BaSSS}
	\label{generalstalk}
	Let $\mcP \in \Spc(\sfT^c)$, and let $a,b \in \sfT^c$. Let $\mcU$ be the collection of all those quasi-compact open subsets of $\Spc(\sfT^c)$ containing $\mcP$. Then there is a natural isomorphism 
	\[\colim_{U\in\mcU}\Hom_{\sfT^c(U)}(a,b) \cong \Hom_{\sfT^c/\mcP}(a,b).\]
\end{lem}

We end this section by recalling some technical results about the Balmer support and prime ideals.

\begin{lem}
	\label{techstuff}
	For an essentially small tensor triangulated category $\sfT^c$ the following hold for all $\mcP\in\Spc(\sfT^c).$
	
	\begin{enumerate}
		\item $\supp(\mcP)=\{\mcQ \in \Spc(\sfT^c)\ \vert \ \mcP\not\subseteq\mcQ\}=\mcZ(\mcP)$.
		\item $\supp(\mcP)=\bigcup_{U(a)\ni\mcP}\supp(a) =\bigcup_{U\ni\mcP}\Spc(\sfT^c)\setminus U$ taken over all quasi-compact opens $U$ containing $\mcP.$
		\item $\sfT^c_{\supp(\mcP)}=\mcP$.
		\item $\mcP=\bigcup_{\mcP\in U}\sfT^c_{\Spc(\sfT^c)\setminus U}$.
		\item For a Thomason subset $Z \subseteq \Spc(\sfT^c)$ we have $\supp \sfT^c_Z \subseteq Z$. Moreover if $\Spc(\sfT^c)$ is noetherian then $\supp \sfT^c_Z = Z$. 
	\end{enumerate}
\end{lem}
\begin{proof}
	
	\begin{enumerate}
		\item We have $\mcS \in \supp(\mcP)$ iff there exists $p \in \mcP$ such that $\mcS \in \supp(p)$, iff $p \not\in \mcS$, i.e $\mcP \not\subseteq \mcS$. The second equality is immediate from the definition of $\mcZ(\mcP)$ and noting that $\mcP\not\subseteq \mcQ$ iff $\mcP$ is not in the closure of $\mcQ.$
		
		\item First note that $\bigcup_{U\ni\mcP}\Spc(\sfT^c)\setminus U=\Spc(\sfT^c)\setminus \bigcap_{U\ni\mcP}U$. By \cite[2.14]{BaSpec} each quasi-compact open is of the form 
		\[U=U(a):=\{\mcQ\in\Spc(\sfT^c) \ \vert \ a \in \mcQ\},\]
		for some compact object $a$. The complement of $U(a)$ is obviously $\supp(a).$ By assumption $\mcP \in U(a)$ so $a \in \mcP$. An ideal $\mcQ \in \bigcup_{U(a)\ni\mcP}\supp(a)$ 
		if and only if there is some $a$ so that $\mcQ\in \supp(a)$, if and only if $\mcQ \in \supp(\mcP)$. Therefore the equality holds.
		
		\item If $t \in \mcP$ then by definition $\supp(t) \subseteq \supp(\mcP)$ and so $t \in \sfT^c_{\supp(\mcP)}$. 
		If $t \in \sfT^c_{\supp(\mcP)}$ then as $\supp(t) \subseteq \supp(\mcP)$ we have $\mcP \not\in \supp(t)$ and so $t \in \mcP$.
		
		\item Let $t \in \mcP$ and let $Z=\supp(t)$, which is closed with open complement $U$. Note that $\mcP \in U$. It is immediate that $t \in \sfT^c_Z$ and so $\mcP \subseteq \bigcup_{\mcP\in U}\sfT^c_Z$. For the other inclusion let $t \in \bigcup_{\mcP\in U}\sfT^c_Z$. Then there exists an open subset $U$ with closed complement $Z$ such that $\mcP \in U$ and $t \in \sfT^c_Z$. That is $\supp(t) \subseteq Z$. Then $\mcP \not\in \supp(t)$ and so $t \in \mcP$. 
		
		\item Fix $\mcP \in \supp \sfT^c_Z$. Then there exists $t \in \sfT^c_Z$ such that $\mcP \in \supp(t)$ from which it is immediate that $\mcP \in Z$. When $\Spc(\sfT^c)$ is noetherian there exist objects $t_\lambda \in \sfT^c$ such that $Z=\bigcup_{\lambda \in \Lambda} \supp(t_\lambda)$. Then for $\mcP\in Z$ there exists $\lambda \in \Lambda$ such that $\mcP \in \supp(t_\lambda)\subseteq Z$, and so $\mcP \in \supp\sfT^c_{\supp(t)}$ as required.
		
	\end{enumerate}
	
\end{proof}

\begin{defn}
	The \emph{radical} $\sqrt{\mcJ}$ of a thick $\otimes$-ideal $\mcJ$ is defined to be
	\[\sqrt{\mcJ}:=\{r \in \sfT \ \vert \ \exists n \geq 1 \text{ such that }r^{\otimes n} \in \mcJ\}.\]
	A thick subcategory $\mcJ$ is called \emph{radical} if $\mcJ=\sqrt{\mcJ}.$ 
\end{defn}

\begin{rem}
	If the category $\sfT^c$ is rigid then every thick $\otimes$-ideal is radical. All thick $\otimes$-ideals are radical if and only if $r \in \thick^{\otimes}(r\otimes r)$ for every object $r \in \sfT^c$ \cite[4.4]{BaSpec}. 
\end{rem}

\begin{prop}\cite[4.9]{BaSpec}
	\label{radicalsupp}
	Let $\mcJ\subset \sfT^c$ be a thick $\otimes$-ideal. Then 
	\[\sfT^c_{\supp(\mcJ)}=\sqrt{\mcJ}.\]
\end{prop}

\section{Associated sheaf functors}

Given a big tt-category $\sfT$ we have now seen the construction of locally ringed spaces $\Spec(\sfT^c)$ and $\Spec^\bullet(\sfT^c).$ Given an action of $\sfT$ on a triangulated category $\sfK$ we will define a "relative sheaf functor", which associates to each object $A \in \sfK$ an $\mcO_\sfT$-module (and in the graded case an $\mcO^\bullet_\sfT$-module). This construction extends Balmer's construction of the locally ringed space structure. We will then show how these sheaves interact with the tt-support theories and under what conditions we can determine their coherence properties.

Throughout this section the action of $\sfT$ on $\sfK$ will be denoted by $\ast$.

\begin{defn}
	Consider the spectrum $\Spc(\sfT^c)$. We define the category $\qbasic(\sfT^c)$ of quasi-compact open subsets of $\Spc(\sfT^c)$, with morphisms given by inclusion.
	Given an abelian category $\sfX$ we define $\qbpresheaf_{\sfX}(\Spc(\sfT^c))$ to be the category of $\sfX$-valued presheaves defined over the basis of quasi-compact opens of $\Spc(\sfT^c)$. That is, the category of contravariant functors
	\[F: \qbasic(\sfT^c)^\op \to \sfX.\]
\end{defn}

\begin{rem}
	As the quasi-compact open subsets of $\Spc(\sfT^c)$ are a basis, every $F \in \qbpresheaf_{\sfX}(\Spc(\sfT^c))$ extends uniquely to a sheaf 
	$F^\# \in \Shv_{\sfX}(\Spc(\sfT^c))$.
	
\end{rem}

\begin{construction}
	Let $U \subseteq \Spc(\sfT^c)$ be a quasi-compact open subset with closed complement $Z$. Define 
	\[\sfT_Z = \loc(\sfT^c_Z)=\loc(\{t \in \sfT^c\ \vert \ \supp(t)\subseteq Z\})\]
	and
	\[\sfT(U)=\sfT/\sfT_Z.\]
	Now define 
	\[\sfK(U)=\sfK/(\sfT_Z \ast \sfK).\]
	If $V$ is a quasi-compact open subset of $U$ then $(\sfK(U))(V)\simeq\sfK(V).$
	
	For a prime ideal $\mcP \in \Spc(\sfT^c)$ we define
	\[\sfK(\mcP)=\sfK/(\loc^\otimes(\mcP)\ast\sfK).\]
\end{construction}

There is an action of $\sfT(U)$ on $\sfK(U)$ induced by the action $\ast$ of $\sfT$ on $\sfK$ \cite{StevensonActions}. We will also use $\ast$ to denote this induced action. 

\begin{rem}
	Instead of restricting to a quasi-compact open, we can instead restrict to the complement of a Thomason subset of the spectrum. That is, given a Thomason subset $Z \subseteq \Spc(\sfT^c)$ with complement $U$ we can repeat the above construction to produce the corresponding categories $\sfT(U)$ and $\sfK(U)$. We focus on the case where $U$ is a quasi-compact open, as in this case we have $\Spc(\sfT(U)^c) \cong U$ \cite[1.11]{bfgluing}. 
\end{rem}

\begin{defn}
	Fix an invertible object $u \in \sfT$. Define a functor 
	\[\prescript{}{\mathrm{p}}{[-,-]}(-):(\sfK^c)^\op \times \sfK \times \qbasic(\Spc(\sfT^c)) \to \gr\Ab\] by 
	\[\prescript{}{\mathrm{p}}{[A,B]}(U)=\Hom^\bullet_{\sfK(U)}(A_U,B_U),\]
	where $A_U,$ $B_U$ are the respective images of $A$ and $B$ under the localisation functor $q_U : \sfK \to \sfK(U),$ and \[\Hom^\bullet_{\sfK(U)}(A_U,B_U)\cong \bigoplus_{i \in \ZZ}\Hom_{\sfK(U)}(A_U,(u^{\otimes i}\ast B)_U)\] as in Definition \ref{defgradedhom}.
	This defines a functor
	\[\prescript{}{\mathrm{p}}{[-,-]}:(\sfK^c)^\op \times \sfK \to \qbpresheaf_{\gr\Ab}(\Spc(\sfT^c)),\]
	which extends uniquely to a functor
	\[[-,-]^\#:(\sfK^c)^\op \times \sfK\to \Shv_{\gr\Ab}(\Spc(\sfT^c)).\]
	We call this the \emph{associated sheaf functor}. We say that $[A,B]^\#$ is the \emph{sheaf associated to } $A$ \emph{and} $B$ \emph{relative to} $u$. 
	
\end{defn}

\begin{prop}
	The functor $[-,-]^\#:(\sfK^c)^\op \times \sfK\to \Shv_{\gr\Ab}(\Spc(\sfT^c))$ upgrades to a functor 
	\[[-,-]^\#:(\sfK^c)^\op \times \sfK\to \grMod\mcO^\bullet_\sfT.\]  
\end{prop}

\begin{proof}
	Fix $(A,B)\in(\sfK^c)^\op \times \sfK$ and a quasi-compact open subset $U\subseteq \Spc(\sfT^c).$ Consider the sections $\prescript{}{\mathrm{p}}{[A,B]}(U)=\Hom^\bullet_{\sfK(U)}(A_U,B_U).$ We define a graded module structure via the action of $\sfT$ on $\sfK$. As noted this action restricts to an action of $\sfT(U)$ on $\sfK(U)$. Fix $f \in \prescript{}{\mathrm{p}}{\mcO^\bullet_\sfT}(U)$ with $\deg(f)=i$ and $g \in \prescript{}{\mathrm{p}}{[A,B]}(U)$ with $\deg(g)=j$. Explicitly $f \in \Hom_{\sfT(U)}(\unit_U,(u^i \otimes \unit)_U)$ and $g \in \Hom_{\sfK(U)}(A_U,(u^j \ast B)_U).$ Define $f \cdot g$ to be the composite
	\[A_U \xrightarrow{\simeq} \unit_U\ast A_U \xrightarrow{f\ast g}u^i_U \ast (u^j_U \ast B_U)\xrightarrow{\simeq}u^{i+j}_U \ast B_U.\]
	This multiplication is compatible with the addition on morphisms, and the grading. Together with the compatibility of the action with the restriction maps this completes the proof.     
\end{proof}

We now compute the stalks of these sheaves. To do this we require that $\Spc(\sfT^c)$ be a noetherian topological space, and that $\sfT$ has a well-behaved theory of homotopy colimits by appealing to some higher structure such as some flavour of model, derivator, $\infty$-category (see \cite[6.5]{StevensonActions} for more discussion). 

\begin{lem}
	\label{lemactionstalk}
	For a prime $\mcP \in \Spc{\sfT^c}$ the stalk of the sheaf $[A,B]^\#$ at $\mcP$ is given by
	\[[A,B]^\#_\mcP=\Hom^\bullet_{\sfK(\mcP)}(A_\mcP,B_\mcP)\]
	where 
	\[\Hom^\bullet_{\sfK(\mcP)}(A_\mcP,B_\mcP)=\bigoplus_{i \in \ZZ}\Hom_{\sfK(\mcP)}(A_\mcP,u^i_\mcP \ast B_\mcP).\]
\end{lem}

\begin{proof}
	We use the identification $\Gamma_\mcV\sfK=\Gamma_\mcV\sfT \ast \sfK = \sfT_\mcV \ast \sfK$ from \cite[4.11]{StevensonActions}, where $\mcV$ is the complement of a quasi-compact open subset $U$. Now consider all quasi-compact opens $U_i$ containing $\mcP$, with complements $\mcV_i$. We first note the natural morphisms induced by inclusions:
	\[
	\begin{tikzcd}
		L_{\mcV_i}B \arrow[rr] \arrow[rd] &                              & L_{\mcZ(\mcP)} B \\
		& L_{\mcV_j}B \arrow[ru] &               
	\end{tikzcd}\]
	which induces a map between $\hocolim L_{\mcV_i}B$ and $L_{Z(\mcP)}B$. This map completes to a triangle 
	\[\hocolim L_{\mcV_i}B \to L_{Z(\mcP)} B \to Z \to \Sigma\hocolim L_{\mcV_i}B.\]
	For each $\mcQ \not\in \mcZ(\mcP)$ we have
	\[\Gamma_\mcQ L_{Z(\mcP)}B\cong \Gamma_{\mcQ}B.\]
	Applying this functor to the above gives a triangle
	\[\Gamma_{\mcQ}\hocolim L_{\mcV_i}B \to \Gamma_{\mcQ}B \to \Gamma_{\mcQ}Z \to \Gamma_{\mcQ}\Sigma\hocolim L_{\mcV_i}B.\]
	Given that $\Gamma_{\mcQ}\hocolim L_{\mcV_i}B \cong \Gamma_{\mcQ}B$ the first morphism in the triangle is an isomorphism and so $\Gamma_{\mcQ}Z \cong 0.$ As $\sfT$ has noetherian spectrum and we assumed a good theory of homotopy colimits this forces $Z\cong 0$ by the local-to-global principle \cite[6.8]{StevensonActions}. Therefore $\hocolim L_{\mcV_i}B \cong L_{Z(\mcP)}B.$ Hence
	\begin{align*}
		[A,B]^\#_\mcP &\cong \colim \Hom^\bullet_{\sfK(U)}(A_U,B_U),\\
		&\cong \colim\Hom^\bullet_\sfK(A,L_{\mcV_i}B)\text{ by adjunction,}\\
		&\cong \Hom^\bullet_\sfK(A,\hocolim L_{\mcV_i}B)\text{ by compactness,}\\
		&\cong \Hom^\bullet_\sfK(A,L_{Z(\mcP)}B)\text{ by the isomorphism in the previous paragraph,}\\
		&\cong \Hom^\bullet_{\sfK(\mcP)}(A_\mcP,B_\mcP)\text{ by adjunction.}
	\end{align*}
\end{proof}

\begin{lem}
	For all $A \in \sfK^c$ the functor $[A,-]^\#$ is homological and coproduct preserving.
\end{lem}

\begin{proof}
	Observe that as $\Hom(A,-)$ is homological and the associated sheaf functor is exact, the functor $[A,-]^\#$ must be homological. 
	To prove coproduct preservation fix a quasi-compact open basis set $U \subseteq \Spc(\sfT^c)$ and some set-indexed collection of objects $(B_\lambda)_{\lambda \in \Lambda}$ in $\sfK$. As the localisation sequence used in the construction of $\sfK(U)$ is smashing \cite{StevensonActions} the canonical functor $\sfK \to \sfK(U)$ preserves compact objects. Therefore there is an isomorphism
	\[\Hom^{\bullet}_{\sfK(U)}(A,\coprod_{\lambda \in \Lambda}B_\lambda)\cong \bigoplus_{\lambda \in \Lambda}\Hom^{\bullet}_{\sfK(U)}(A,B_\lambda)\]
	and so it follows there is an isomorphism
	\[\prescript{}{\mathrm{p}}{[}A,\CMcoprod_{\lambda \in \Lambda}B_\lambda]\cong \bigoplus_{\lambda \in \Lambda}\prescript{}{\mathrm{p}}{[}A,B_\lambda]\]
	and so 
	\[
	[A,\CMcoprod_{\lambda \in \Lambda}B_{\lambda}]^{\#} \cong
	(\bigoplus_{\lambda \in \Lambda}\prescript{}{\mathrm{p}}{[}A,B_\lambda])^\#.
	\]
	
	Sheafification of a presheaf is a left adjoint and so preserves coproducts. Therefore 
	\begin{align*}
		[A,\CMcoprod_{\lambda \in \Lambda}B_\lambda]^\#&\cong \bigoplus_{\lambda \in \Lambda}\prescript{}{\mathrm{p}}{[}A,B_\lambda]^\#\\
		&\cong\bigoplus_{\lambda \in \Lambda}[A,B_\lambda]^\#.
	\end{align*}
\end{proof}

We organise this information into the following definition.

\begin{defn}
	Fix an action of a big tt-category $\sfT$ on a triangulated category $\sfK$. Let $u$ be an invertible object in $\sfT^c$ and $A$ an invertible object in $\sfK^c$. For an object $B \in \sfK$ we define the $u$\emph{-twisted } $A$\emph{-support} of $B$ by
	\[\supp^\bullet(B ,A)=\{\mcP \in \Spc(\sfT^c) \ \vert \ [A,B]^{\#}_\mcP \neq 0 \}.\]
	When considering the untwisted case we write $\supp(B,A)$. 
\end{defn}

We note the following basic properties:

\begin{prop}
	\label{suppbasic}
	Given the notion of support defined above we have:
	\begin{enumerate}
		\item $\supp^\bullet(0,A) = \varnothing$.
		\item Given a triangle $B\to C \to D \to \Sigma B$ in $\sfK$ we have $\supp^\bullet(C,A)\subseteq \supp^\bullet(B,A)\cup\supp^\bullet(D,A)$.
		\item $\supp^\bullet(B,A) = \supp^\bullet(u\ast B,A),$
		\item $\supp^\bullet(B, u^\vee \ast A)=\supp^\bullet(B,A)$, where $u^\vee$ is the dual of $u$.
		\item $\supp^\bullet(\bigoplus_{\lambda \in \Lambda}B_\lambda,A)=\bigcup_{\lambda \in \Lambda}\supp^\bullet(B,A)$,
		\item $\supp^\bullet(B,\bigoplus_{\lambda \in \Lambda}A_\lambda)=\bigcup_{\lambda \in \Lambda}\supp^\bullet(B,A_\lambda)$
	\end{enumerate}
\end{prop}

\begin{proof}
	\begin{enumerate}
		\item As $[A,0]^\#=0$, all of the stalks vanish at each $\mcP$ and so the support is empty as required.
		\item Let $\mcP \in \supp^\bullet(C,A)$. Applying the homological functor $[A,-]^\#$ followed by the stalk functor at $\mcP$ gives a long exact sequence
		\[\cdots \to [A,B]^{\#}_\mcP\to [A,C]^{\#}_\mcP\to [A,D]^{\#}_\mcP\to [A,\Sigma B]^{\#}_\mcP\to\cdots\]
		If $\mcP \not\in \supp^\bullet(B,A)\cup\supp^\bullet(D,A)$ then as the sequence is exact we would have $[A,C]^{\#}_\mcP=0$, a contradiction. Therefore $\mcP \in \supp^\bullet(B,A)\cup\supp^\bullet(D,A)$ and the result holds.
		\item We have the following chain of if and only ifs:
		\begin{align*}
			\mcP \in \supp^\bullet(B,A) &\Leftrightarrow \Hom_{\sfK(\mcP)}(A,u^{\otimes i}\ast B) \neq 0 \text{ for some }i \in \ZZ,\\
			&\Leftrightarrow \Hom_{\sfK(\mcP)}(A,u^{\otimes i-1}\ast(u\ast B)) \neq 0 \text{ for some }i \in \ZZ,\\
			&\Leftrightarrow \mcP \in \supp^\bullet(u \ast B, A).
		\end{align*} 
		
		\item Similarly to the above we have the following chain of if and only ifs:
		\begin{align*}
			\mcP \in \supp^\bullet(B,u^\vee \ast A) &\Leftrightarrow \Hom_{\sfK(\mcP)}(u^\vee \ast A,u^{\otimes i}\ast B) \neq 0 \text{ for some }i \in \ZZ,\\
			&\Leftrightarrow \Hom_{\sfK(\mcP)}(A,u^\vee \ast u^{\otimes i}\ast B) \neq 0 \text{ for some }i \in \ZZ,\\
			&\Leftrightarrow \Hom_{\sfK(\mcP)}(A,u^{\otimes i-1}\ast B) \neq 0 \text{ for some }i \in \ZZ,\\
			&\Leftrightarrow \mcP \in \supp^\bullet(B, A).
		\end{align*}
		\item As $A$ is compact we have an isomorphism 
		\[\Hom^\bullet_{\sfK(\mcP)}(A,\bigoplus_{\lambda \in \Lambda}B_\lambda) \cong \bigoplus_{\lambda \in \Lambda}\Hom^\bullet_{\sfK(\mcP)}(A,B_\lambda).\]
		Using this gives us another chain of implications:
		\begin{align*}
			\mcP \in \supp^\bullet(\bigoplus_{\lambda \in \Lambda}B_\lambda, A) &\Leftrightarrow \Hom^\bullet_{\sfK(\mcP)}(A,\bigoplus_{\lambda \in \Lambda}B_\lambda)\neq 0\\
			&\Leftrightarrow \bigoplus_{\lambda \in \Lambda}\Hom^\bullet_{\sfK(\mcP)}(A,B_\lambda) \neq 0\\
			&\Leftrightarrow\mcP \in \supp(B_\lambda, A) \text{ for some }\lambda \in \Lambda\\
			&\Leftrightarrow\mcP \in \bigcup_{\lambda \in \Lambda}\supp^\bullet(B_\lambda, A).\\  
		\end{align*}
		\item Almost identical to the above, using properties of the hom-functor.
	\end{enumerate}
\end{proof}

\begin{rem}
	\label{untwisteddecomp}
	Consider the twisted support $\supp^\bullet(B, A)$ with respect to an invertible object $u$. We can decompose this twisted support into a union of untwisted supports. By definition $\Hom_{\sfK(\mcP)}^\bullet (A, B) \neq 0$ if and only if there exists $i \in \ZZ$ such that $\Hom_{\sfK(\mcP)}(A, u^{\otimes i}\ast B) \neq 0$. Therefore $[A,B]^{\#}_\mcP \neq 0$ if and only if there exists $i \in \ZZ$ such that $[A,u^{\otimes i}\ast B]^{\#}_\mcP \neq 0$. Note that the first sheaf if twisted with respect to $u$ while the second is untwisted. Taking the union over all $i \in \ZZ$ we obtain
	\[\supp^\bullet(B,A)=\bigcup_{i \in \ZZ}\supp(u^{\otimes i}\ast B, A).\]  
\end{rem}

\section{Comparison of support theories}

We can now compare our notion of support with the usual definitions of support found in tt-geometry. 

\begin{defn}
	Let $\sfT$ act on $\sfK$. Let $B \in \sfK$ be an object with non-empty support. We say an $B$ is \emph{locally generated} if for each $\mcP\in \supp(B)$ there exists a compact object $A \in \sfK$ such that $B \in \loc(A_\mcP)$,  realised as a full subcategory of the Verdier quotient $\sfK/\loc(\mcP)$. We say a function $\phi: \supp(B) \to \sfK^c$ is a \emph{generating function} if $B$ is locally generated by the collection of compact objects $\{\phi(\mcP)\ \vert \ \mcP\in \supp(B)\}$.
\end{defn}

\begin{lem}
	\label{ttinsheaf}
	Let $\sfT$ act on $\sfK$. Let $B$ be a locally generated object of $\sfK$ with generating function $\phi$. Then there exists a function $\psi: \supp(B) \to \ZZ$ such that
	\[\supp(B) \subseteq \bigcup_{\mcP\in \supp(B)} \supp^\bullet(\Sigma^{\psi(\mcP)} B,\phi(\mcP)). \]
	In particular when $u=\Sigma \unit$ we have
	\[\supp(B) \subseteq \bigcup_{\mcP\in \supp(B)}\supp^\bullet(B,\phi(\mcP)).\] 
\end{lem}

\begin{proof}
	The result is trivial if the support of $B$ is empty. Assuming that the support is non-empty choose a prime ideal $\mcP \in \supp(B)$. Note by the adjunction in the associated localisation sequence, the localisation $B_\mcP\in \sfK/\loc(\mcP)$ is non-zero. As $B_\mcP\in \loc(\phi(\mcP))$, by Lemma \ref{lemthickzero} there exists $i \in \ZZ$ such that $\Hom_{\loc(\phi(\mcP))}(\phi(\mcP)_\mcP, \Sigma^i B_\mcP)\neq 0.$ As $\loc(\phi(\mcP))$ is full we have $\Hom_{\loc(\phi(\mcP))}(\phi(\mcP)_\mcP, \Sigma^i B_\mcP)=\Hom_{\sfK(\mcP)}(\phi(\mcP)_\mcP, \Sigma^i B_\mcP)$ and so $\Hom_{\sfK(\mcP)}(\phi(\mcP)_\mcP, \Sigma^i B_\mcP) \neq 0.$ Therefore $[\phi(\mcP),\Sigma^i B]^\#_\mcP = \Hom_{\sfK(\mcP)}(\phi(\mcP)_\mcP, \Sigma^i B_\mcP)\neq 0,$ and therefore $\mcP \in \supp^\bullet(\Sigma^i B, \phi(\mcP)).$ Defining $\psi(\mcP)=i$, and then taking the union over all $\mcP \in \supp(B)$ completes the proof. The second statement follows immediately from Proposition \ref{suppbasic}.   
\end{proof}

\begin{rem}
	\label{gencon}
	If $\sfK=\loc(A)$ for some compact object $A$ then for any object $B$ the constant function $\phi:\supp(B) \to \{A\}$ is a generating function. Then the previous lemma can be simplified by replacing all instances of $\phi(\mcP)$ with $A$. In particular when $u=\Sigma \unit$ we have
	\[\supp(B) \subseteq \supp^\bullet(B, A).\]
\end{rem}

A natural place to compare support theories is for the case of a tt-category $\sfT$ acting on itself via the tensor product. The support theory defined by Stevenson in this setting agrees with that of Balmer when restricted to compact objects of $\sfT$. The following lemma pushes our theory of support in the same direction.

\begin{lem}
	\label{sheafintt}
	For any $A,B \in \sfT^c$ we have 
	\[\supp^\bullet(B,A) \subseteq \supp B.\]
\end{lem}

\begin{proof}
	Fix $\mcP \in \supp^\bullet(B,A).$ By definition
	\[\Hom^\bullet_{\sfT^c/\mcP}(A,B)\cong [A,B]^\#_\mcP \neq 0.\]
	If $B \in \mcP$ then $B \cong0$ in $\sfT^c/\mcP$. But then the hom-group would be $0$, a contradiction.
\end{proof}

\begin{thm}
	\label{thicksupcompare}
	If $\sfT^c=\thick(A)$ then for all objects $B \in \sfT^c$ there is an equality
	\[\supp B = \bigcup_{i\in\ZZ}\supp^\bullet (\Sigma^i B, A).\]
	When $u \cong \Sigma \unit$ we obtain
	\[\supp B = \supp^\bullet(B,A).\]
\end{thm}

\begin{proof}
	As $\sfT^c = \thick(A)$ by Remark \ref{gencon} and Lemma \ref{ttinsheaf} we have
	\[\supp(B) \subseteq \bigcup_{\mcP\in \supp(B)} \supp^\bullet(\Sigma^{\psi(\mcP)} B,A)\subseteq \bigcup_{i\in\ZZ}\supp^\bullet (\Sigma^i B, A),\]
	which gives the first inclusion. For the reverse inclusion we have 
	\begin{align*}
		\supp^\bullet(\Sigma^i B, A) &\subseteq \supp(\Sigma^i B)\text{ by Lemma \ref{sheafintt},}\\
		&\subseteq \supp(B).
	\end{align*}
	And so we obtain the first equality. The second equality follows immediately as when $u\cong \Sigma \unit$ we have $\supp^\bullet(\Sigma^i B,A)=\supp(B,A).$  
\end{proof}

\section{Almost thick preimages of quasi-coherent sheaves}
We show that in general the collection of objects which are quasi-coherent after sheafification is almost a thick subcategory, with a clear potential obstruction. The proof only relies on the formal property of being a coproduct preserving homological functor and so the intermediate steps are simple. 

\begin{convention}
	\label{convfuncB}
	Throughout this section $\bfB:\sfT \to \sfA$ will be a homological, coproduct preserving functor from a big tt-category $\sfT$ to an abelian category $\sfA$. $\sfX$ will be an additive subcategory of $\sfA$. 
	We denote the preimage of $\sfX$ under $\bfB$ by
	\[\bfB^{-1}(\sfX)=\{C \in \sfT\ \vert \ \bfB(C)\in \sfX\}\]
	as a full subcategory of $\sfT.$
\end{convention}

We note the following elementary properties of $\bfB^{-1}(\sfX)$ in the following series of propositions. Proofs are included for completeness. 

\begin{prop}
	\label{basicinverse}
	We have
	\begin{enumerate}
		\item $\bfB^{-1}(\sfX)$ is an additive subcategory of $\sfT.$
		\item If $\sfX$ is closed under summands then $\bfB^{-1}(\sfX)$ is closed under summands.
		\item If $\sfX$ is cocomplete then $\bfB^{-1}(\sfX)$ is cocomplete.
	\end{enumerate}
\end{prop}

\begin{proof}
	\begin{enumerate}
		\item $\sfX$ is additive so contains the zero object $0_\sfA.$ As $\bfB$ is additive we have $\bfB(0_\sfT)=0_\sfA$ and so $0_\sfT \in \bfB^{-1}(\sfX).$ The preimage is  also closed under finite coproducts. Given $A, B \in \bfB^{-1}(X)$ we have $\bfB(A\oplus B)\cong \bfB(A) \oplus \bfB(B) \in \sfX$, from which the result follows.
		\item Consider an object $A \in \bfB^{-1}(\sfX)$ such that $A\cong B\oplus C.$ By additivity $\bfB(A)\cong \bfB(B)\oplus \bfB(C).$ As $\sfX$ is closed under summands both $\bfB(B)$ and $\bfB(C)$ are in $\sfX$ and so $B,C \in \bfB^{-1}(\sfX)$ as required.
		\item Let $Y_\lambda$ be a collection of objects in $\bfB^{-1}(\sfX)$. As $\sfX$ is cocomplete it contains the coproduct $\coprod_{\lambda}\bfB(Y_\lambda)$. $\bfB$ is coproduct preserving so $\coprod_{\lambda}\bfB(Y_\lambda)\cong \bfB(\coprod_{\lambda}Y_\lambda)$. Then $\coprod_{\lambda}Y_\lambda\in\sfX$ as required.
	\end{enumerate}
\end{proof}

\begin{rem}
	Every abelian subcategory is closed under summands, as each summand is the kernel of a projection.
\end{rem}

\begin{defn}
	A category is \emph{wide} if it is abelian and closed under extensions.
\end{defn}

\begin{prop}
	\label{wide}
	Suppose $\sfX$ is a wide subcategory and consider an exact triangle \[X\to Y \to Z \to \Sigma X\] such that all but $Y$ are assumed to be in $\bfB^{-1}(\sfX).$ If $\Sigma^{-1}Z\in \bfB^{-1}(\sfX)$ then $Y \in \bfB^{-1}(\sfX).$
\end{prop}

\begin{proof}
	Apply $\bfB$ to the distinguished triangle given to obtain a long exact sequence
	\[\cdots \to \bfB(\Sigma^{-1}Z)\xrightarrow{e}\bfB(X)\xrightarrow{f}\bfB(Y)\xrightarrow{g}\bfB(Z)\xrightarrow{h}\bfB(\Sigma X)\to \cdots\]
	We have a short exact sequence 
	\[0\to \coker(e) \to \bfB(Y) \to \ker(h) \to 0.\]
	As $\sfX$ is abelian and both $e$ and $h$ are morphisms in $\sfX$, both $\coker(e)$ and $\ker(h)$ are in $\sfX$. As $\sfX$ is also closed under extensions, this forces $\bfB(Y)\in \sfX$ as required.
\end{proof}
Combining the above propositions gives the following corollary:

\begin{cor}
	If $\sfX$ is wide and $\bfB^{-1}(\sfX)$ is closed under suspensions, then $\bfB^{-1}(\sfX)$ is thick.
\end{cor}

This continues our general theme of needing additional suspension data to obtain results. We now show that provided we have the right data on suspensions we can just focus on a collection of generators. 

\begin{construction}
	Given a collection of objects $\mcX$ we can generate the full subcategory $\thick(\mcX)$ iteratively. Set
	\begin{align*}
		\langle \mcX\rangle_1&=\add(\Sigma^k X \ \vert \ X \in \mcX, k \in \ZZ)\\
		\langle \mcX \rangle_{i+1} &= \add(X \ \vert \exists\ X_1\to X \to X_2 \to \Sigma X_1,\ X_1\in\langle \mcX\rangle_i, X_2\in\langle \mcX\rangle_1)
	\end{align*}
	Then $\thick(\mcX)=\cup_{i=1}^\infty \langle \mcX \rangle _i.$ Note that if $X \in \langle \mcX \rangle_i$, then so is $\Sigma^k X$ for all $k \in \ZZ$.
\end{construction} 

\begin{lem}
	Suppose $\sfX$ is wide. Given a finite set of objects $\mcX=\{x_1,\dots,x_n\}$ in $\sfT$, suppose that for all $k \in \ZZ$ and $j=1,\dots,n$ we have $\Sigma^k x_j \in \bfB^{-1}(\sfX).$ Then
	\[\thick(\mcX)\subseteq \bfB^{-1}(\sfX).\]
\end{lem}

\begin{proof}
	We show that each of the $\langle \mcX \rangle_i$ are contained in $\bfB^{-1}(\sfX)$ by induction. If $Y \in \langle \mcX\rangle_1$ then it is a summand of a coproduct of objects in $\bfB^{-1}(\sfX)$ and so is also in $\bfB^{-1}(\sfX)$ by Proposition \ref{basicinverse}.
	The claim therefore holds for $i =1.$ Now suppose that the claim holds for all $i\leq m.$ Fix $Y \in \langle \mcX \rangle_{m+1}.$ We have a distinguished triangle
	\[X\to Y\oplus W \to Z \to \Sigma X\]
	with $X \in \langle \mcX \rangle_m$ and $Z \in \langle\mcX \rangle_1$. By induction both $X$ and $\Sigma X$ are in $\bfB^{-1}(\sfX)$ and so are both $Z$ and $\Sigma^{-1} Z.$ Therefore so is $Y\oplus W$ by Proposition \ref{wide}. $\sfX$ is wide so Proposition \ref{basicinverse}
	tells us $\bfB^{-1}(\sfX)$ is closed under summands, and so $Y \in \bfB^{-1}(\sfX)$. 
\end{proof}

As an aside we have the following interaction with certain homotopy colimits. 

\begin{prop}
	Suppose $\sfX$ is cocomplete and wide. Let \[Y_0 \xrightarrow{j_1} Y_1 \xrightarrow{j_2} Y_2 \xrightarrow{j_3} \cdots\]
	be a sequence in $\sfT$. If there is an increasing sequence of integers 
	\[0\leq i_0 < i_1 < i_2 < i_3 < \cdots\]
	such that for all $i_k$ we have that $Y_{i_k}$ and $\Sigma Y_{i_k}$ are objects of $\bfB^{-1}(\sfX)$, then $\hocolim Y_i$ is an object of $\bfB^{-1}(\sfX).$
\end{prop}

\begin{proof}
	By Proposition \ref{basicinverse}
	the coproduct of the $Y_{i_k}$ lies in $\bfB^{-1}(\sfX)$ as is the coproduct of the $\Sigma Y_{i_k}$. The homotopy colimit $\hocolim Y_{i_k}$ is defined by the triangle
	\[\coprod_{k=0}^{\infty}Y_{i_k} \xrightarrow{1-shift}\coprod_{k=0}^{\infty}Y_{i_k} \xrightarrow{} \hocolim{Y_{i_k}} \to \Sigma\coprod_{k=0}^{\infty}Y_{i_k}\]
	where the shift map is the direct sum of the maps $j_{i_k +1}: X_{i_k} \to X_{i_k +1}$.
	By Proposition \ref{wide} $\hocolim Y_{i_k}$ lies in $\bfB^{-1}(\sfX)$. By \cite[1.7.1]{NeeCat} $\hocolim Y_i \cong \hocolim Y_{i_k}$, completing the proof.  
\end{proof}

Before applying these properties to some categories of interest let us recall some sheaf-theoretic definitions.

\begin{defn}
	A sheaf of modules $\mcF$ on a ringed space $(X,\mcO_X)$ is of \emph{finite type} if for every $x \in X$ there exists some open neighbourhood $U$ such that $\mcF\vert_U$ is generated by finitely many sections.
\end{defn}

\begin{defn}
	A sheaf of modules $\mcF$ on a ringed space $(X,\mcO_X)$ is \emph{quasi-coherent} if it is locally the cokernel of a map of free modules. That is, there is an open cover $\{U_\lambda\}_{\lambda\in\Lambda}$ of $X$ such that for every $\lambda$ there exist index sets $I_\lambda$ and $J_\lambda$ and an exact sequence of sheaves of $\mcO_X$-modules of the form
	\[\mcO^{\oplus I_\lambda}_{U_\lambda} \to\mcO^{\oplus J_\lambda}_{U_\lambda}\to \mcF\vert_{U_\lambda}\to 0. \]
	A sheaf of $\mcO_X$-modules $\mcF$ is said to be \emph{coherent} if
	\begin{itemize}
		\item $\mcF$ is of finite type and,
		\item for every open $U\subseteq X$ and every finite collection $s_i \in \mcF(U),i=1,\dots,n$ the kernel of the associated map $\mcO^{\oplus n}_U \to \mcF\vert_U$ is of finite type.
	\end{itemize}
\end{defn} 

Note that every coherent sheaf is quasi-coherent.

\begin{prop}
	Let $\sfT$ be a big tt-category acting on a triangulated category $\sfK$. For each $A \in \sfK^c$ consider the associated sheaf functor 
	\[[A,-]^\#: \sfK \to \Modu\mcO_\sfT.\]
	Then this functor fits into the setup of Convention \ref{convfuncB}, taking $\sfA=\Modu\mcO_\sfT$ and $\sfX=\Coh(\Spec(\sfT)).$ Thus the results requiring $\sfX$ abelian hold. Moreover if $\Spec(\sfT^c)$ is a scheme, setting $\sfX=\QCoh(\Spec(\sfT))$ will also satisfy the required conditions.
\end{prop}

\begin{rem}
	Note that this setup holds for both the twisted and non-twisted case as appropriate.
\end{rem}

We can condense this section into the following result:

\begin{cor}
	Suppose $\sfK=\thick(X).$ If for all $i \in \ZZ, [A,\Sigma^i X]^\#$ is coherent, then for all $B \in \sfK$ the sheaf $[A,B]^\#$ is coherent.
	If $\Spec(\sfT^c)$ is a scheme then if for all $i \in \ZZ, [A,\Sigma^i X]^\#$ is quasi-coherent, then for all $B \in \sfK$ the sheaf $[A,B]^\#$ is quasi-coherent.
\end{cor}

\begin{cor}
	Suppose $\sfK=\thick(\unit)$ and let $u= \Sigma \unit.$ If the $\Sigma \unit$-twisted structure sheaf $\mcO^\bullet_{\sfT}$ is coherent, then every sheaf of the form $[\unit, B]^\#$ is coherent. 
\end{cor}

\begin{proof}
	Follows immediately from the previous corollary: if the $\Sigma \unit$-twisted structure sheaf $\mcO^\bullet_{\sfT}$ is coherent then each of the sheaves $[\unit, \Sigma^i \unit]^\#$ are coherent. Therefore by the previous corollary each sheaf of the form $[\unit, B]^\#$ is coherent. 
\end{proof}

\section{Affine categories and quasi-coherence}
Our results on (quasi-)coherence so far require us to already know the nature of the suspensions of generators. 
We now restrict to the case of $\sfT$ acting on itself, and consider sheaves of the form $[\unit,X]^\#.$ We will show that when the Balmer spectrum of $\sfT$ is well-behaved, all such sheaves are quasi-coherent.

\begin{defn}
	We say a big tt-category $\sfT$ is \emph{affine} if the natural comparison map $\rho:\Spec(\sfT^c) \to \Spec(R_\sfT)$ is an isomorphism. 
	We say that $\sfT$ is \emph{schematic} if there exists an open cover $\{U_i\}$ of $\Spc(\sfT^c)$ such that for each $i$ the natural comparison maps $\rho_i:\Spec(\sfT^c(U_i))\to \Spec(R_{\sfT(U_i)})$ are isomorphisms. 
	In the twisted case we instead consider the corresponding graded morphisms and say $\sfT$ is \emph{twisted affine} and \emph{twisted schematic} respectively.  
\end{defn}

Immediately from the definitions we see that every affine category is schematic and that the spectrum of a schematic category is a scheme.

\begin{prop}
	If $X$ is a quasi-compact quasi-seperated scheme then $\mcD^{\perf}(X)$ is a schematic category. 
\end{prop}

\begin{proof}
	As $X$ is a scheme we have an open cover of $X$ by affines $U_i.$ Let $E(U_i)$ be the image of $U_i$ under the homeomorphism $X\cong \Spc(\mcD^{\perf}(X))$ of \cite[7.3]{BaPresheaves}. Then
	\[E(U_i) \cong \Spc(\mcD^{\perf}(X))(E(U_i))\cong \Spc(\mcD^{\perf}(U_i))\]
	by \cite[7.8]{BaPresheaves}. By \cite[8.1]{BaSSS} each of the $\mcD^{\perf}(U_i)$ are affine, completing the proof.
\end{proof}

A category being affine is sufficient for all of the associated sheaves $[\unit, X]^\#$ to be quasi-coherent. We start by considering compact objects.

\begin{prop}
	\label{propcomsec}
	For a compact object $X \in \sfT^c$ and any endomorphism $s \in R_\sfT$ we have
	\[
	\prescript{}{\mathrm{p}}{[}\unit,X](U(s))\cong \Hom_{\sfT^c}(\unit, X)[s^{-1}].\]
\end{prop} 

The proof of this proposition is nearly identical to \cite[6.9]{BaSSS}. We will fill in those details left to the reader in \cite{BaSSS}. 

\begin{proof}
	By definition 
	\[\prescript{}{\mathrm{p}}{[}\unit, X](U(s))=\Hom_{\sfT^c/\sfT^c_Z}(\unit,X)
	\]
	where $Z$ is the closed complement of $U(s)$. In fact $Z=\supp(\cone(s))$ and clearly
	\[\supp(\cone(s))=\supp(\thick^\otimes(\cone(s))).\]
	Applying Proposition \ref{radicalsupp}
	\[\sfT^c_Z = \sqrt{\thick^\otimes(\cone(s))}.
	\]
	As $\sfT^c$ is rigid this gives $\sfT^c_Z = \thick^\otimes(\cone(s)).$ By \cite[2.16]{BaSSS} we have
	\[\thick^\otimes(\cone(s)) = \thick^\otimes(\cone(s^i)\ \vert \ i\geq 0).\]
	This is the ideal $\mcJ$ associated to the multiplicative set $S=\{s^i \ \vert \ i\geq0\}$ as in \cite[3.5]{BaSSS}. Applying \cite[3.6]{BaSSS} gives
	\[\prescript{}{\mathrm{p}}{[}\unit, X](U(s))\cong\Hom_{\sfT^{c}/\sfT^{c}_Z}(\unit, X)\cong\Hom_{\sfT^c/\mcJ}(\unit, X) \cong S^{-1}\Hom_{\sfT^c}(\unit, X),
	\]
	and so $\prescript{}{\mathrm{p}}{[}\unit, X](U(s))\cong \Hom_{\sfT^c}(\unit, X)[s^{-1}]$ as required.
\end{proof}

\begin{prop}
	Suppose $\sfT$ affine. Then for every $X \in \sfT^c$ the sheaf $[\unit, X]^\#$ is quasi-coherent on $\Spec(\sfT)$. Explicitly
	\[[\unit, X]^\# \cong \tilde{M}\]
	where $\tilde{M}$ is the sheaf associated to the $R_\sfT$-module $M=\Hom_{\sfT^c}(\unit, X).$
\end{prop}

\begin{proof}
	By Proposition \ref{propcomsec} we have for each $s \in R_\sfT$ an isomorphism
	\[\prescript{}{\mathrm{p}}{[}\unit,X](U(s)) \cong \Hom_{\sfT^c}(\unit, X)[s^{-1}].
	\]
	By \cite[Ch.2, 5.1]{Ha} we also have that
	\[\tilde{M}(D(s))\cong \Hom_{\sfT^c}(\unit,X)[s^{-1}].
	\]
	Given that $\rho$ is a homeomorphism, $D(s)\cong U(s)$ and so $\prescript{}{\mathrm{p}}{[}\unit, X]$ is isomorphic to $\tilde{M}$ on a basis of quasi-compact opens. Therefore the associated sheaves are isomorphic which gives $[\unit, X]^\# \cong \tilde{M}.$ As sheaves associated to modules are quasi-coherent, this completes the proof.
\end{proof}

The result can be extended to non-compact objects using the following theorem.

\begin{thm}\cite[3.3.7]{HPS}
	\label{hpstheorem}
	Let $\sfT$ be a big tt-category. 
	Then given a morphism $s \in R_\sfT$, the thick tensor ideal $\mcJ=\thick^\otimes(\cone(s))$ fits into a localisation diagram
	\[
	\begin{tikzcd}
		\mcJ \arrow[rr, hook] \arrow[d, hook]                              &  & \sfT^c \arrow[rr] \arrow[d, hook]                                                       &  & \sfT^c/\mcJ \arrow[d, hook]               \\
		\loc(\mcJ) \arrow[rr, shift left=3] \arrow[phantom,rr, "\perp" description] &  & \sfT \arrow[rr, shift left=3] \arrow[ll, shift left=3] \arrow[phantom,rr, "\bot" description] &  & \sfT/\loc(\mcJ) \arrow[ll, shift left=3]
	\end{tikzcd}
	\]
	such that for all objects $X \in \sfT$ we have
	\[\Hom_{\sfT/\loc(\mcJ)}(\unit,X)\cong S^{-1}\Hom_\sfT(\unit, X)
	\]
	where $S=\{s^i\ \vert \ i\geq 0\}.$
\end{thm}

\begin{rem}
	The existence of the localisation diagram can actually be obtained from Theorem \ref{millerneeman}.
\end{rem}

Immediately from the theorem we obtain:
\begin{cor}
	For any object $X \in \sfT$ and any endomorphism $s \in R_\sfT$ we have
	\[\prescript{}{\mathrm{p}}{[}\unit, X](U(s))\cong \Hom_\sfT(\unit, X)[s^{-1}].
	\]
\end{cor}

\begin{cor}
	Suppose $\sfT$ affine. Then for every $X \in \sfT$ the sheaf $[\unit, X]^\#$ is quasi-coherent on $\Spec(\sfT).$ Explicitly 
	\[[\unit, X]^\# \cong \tilde{M},\]
	where $\tilde{M}$ is the sheaf associated to the $R_\sfT$-module $M=\Hom_\sfT(\unit, X).$
\end{cor}

The proof of this is identical to the compact case.

We can now conclude with the schematic case.

\begin{lem}
	\label{localres}
	Let $U \subseteq \Spc(\sfT^c)$ be a quasi-compact open subset. Then for objects $A \in \sfT^c, B \in \sfT$ there is an isomorphism of sheaves on $U$:
	\[[A,B]^\#\vert_{U} \cong [A_{U},B_{U}]^\#,\]
	where the associated sheaf functor on the left is over $\Spc(\sfT^c)$ and the associated sheaf functor on the right is over $\Spc(\sfT^c(U)) \cong U$. 
\end{lem}

\begin{proof}
	It suffices to show that the sheaves agree on all quasi-compact open subsets $V \subseteq U$, from which the uniqueness of sheafification will complete the proof.
	Let $V \subseteq U$ be such a quasi-compact open. Recall that $\sfT(V) \simeq (\sfT(U))(V)$. Therefore
	\begin{align*}
		[A_U, B_U]^\#(V)&=\Hom^\bullet_{(\sfT(U))(V)}((A_U)_V,(B_U)_V)\\
		&\cong \Hom^\bullet_{\sfT(V)}(A_V,B_V)\\
		&=[A,B]^\# (V)\\
		&=[A,B]^\#\vert_{U}(V), 
	\end{align*}
	where the last two equalities hold as both $V$ and $U$ are quasi-compact basic open subsets. 
\end{proof}

\begin{thm}
	Let $\sfT$ be schematic. Then for every $X \in \sfT$, the sheaf $[\unit, X]^\#$ is quasi-coherent on $\Spec(\sfT).$
\end{thm}

\begin{proof}
	As $\sfT$ is schematic we are given an open cover $\{U_i\}$ of $\Spc(\sfT^c).$ Fixing $i$, restricting from $\Spc(\sfT^c)$ to $U_i$ and applying the affine result for the affine category $\sfT(U_i)$ and using Lemma \ref{localres} we have
	\[[\unit,X]^\#\vert_{U_i} \cong [\unit_{U_i},X_{U_i}]^\# \cong \tilde{M}_i,\]
	where $M_i = \Hom_{\sfT(U_i)}(\unit, X).$ This is the definition of a quasi-coherent sheaf on a scheme, completing the proof.
\end{proof}

The definitions of affine and schematic provide a coarse grading through which to examine tensor triangular categories. Taking further inspiration from algebraic geometry, we can attach additional conditions to the schematic structure. In particular we translate the notion of being quasi-affine to the tensor-triangulated setting.

\begin{defn}
	Let $\sfT$ be a big tt-category. We say that $\sfT$ is \emph{quasi-affine} if there is an open cover $(U_i)_{i \in I}$ of $\Spc(\sfT^c)$ such that 
	\begin{enumerate}
		\item the cover $(U_i)_{i\in I}$ realises $\sfT$ as a schematic category,
		\item for each $i \in I$ there exists a morphism $s_i \in R_\sfT$ such that $U_i=U(\cone(s_i))$.
	\end{enumerate}
\end{defn}

\begin{ex}
	To motivate the definition observe that if $(X,\mcO_X)$ is a quasi-affine scheme, then the tt-category $\sfD^{\perf}(X)$ is quasi-affine. Indeed, as $X$ is quasi-affine, the structure sheaf $\mcO_X$ is ample and so for each point $x \in X$ there exists an element $s \in \Gamma(X, \mcO_X)$ such that $x \in X_s$. Now by the main result of \cite{BaPresheaves} there is an isomorphism $(\sfD^{\perf}(X), \mcO_{\sfD^{\perf}(X)}) \cong (X, \mcO_X)$ of schemes, and that the associated comparison map is an isomorphism when restricted to affine subsets. The collection of the $X_s$, each of which are affine, therefore form a cover realising the schematic structure of $\sfD^{\perf}(X)$. Now there is an equality \[\Gamma(X, \mcO_X)=\Hom_{\sfD^{\perf}(X)}(\mcO_X, \mcO_X)=R_{\sfD^{\perf}(X)}\]
	and moreover $X_s \cong U(\cone(s))$. Therefore the cover realising the schematic property is of the required form and we conclude that $\sfD^{\perf}(X)$ is quasi-affine.   
\end{ex}

\begin{prop}
	Let $\sfT$ be quasi-affine. Then the collection
	\[\{U(\cone(s)) \ \vert \ s \in R_\sfT)\}\]
	is a basis for the usual topology on $\Spc(\sfT^c)$.
\end{prop}

\begin{proof}
	For a morphism $f$ in $\sfT^c$ we will write $U(f)=U(\cone(f))$. As $\sfT$ is quasi-affine there is a collection of elements $\mcS=(s_i)_i \subseteq R_\sfT$ such that for each index $i$ we have
	\begin{enumerate}
		\item $\Spc(\sfT^c)=\bigcup_{i}U(s_i)$\\
		\item The comparison map $\Spc(\sfT^c(U(s_i)))\xrightarrow{\rho_{U(s_i)}} \Spec(R_{\sfT(U(s_i))})$ is an isomorphism.
	\end{enumerate} 
	Fix the index $i$. Let $S=\{\id, s_i, s^2_i,...\}$ be the multiplicative set generated by $s_i$. By Proposition \ref{propcomsec} we have
	\[R_{\sfT(U(s_i))}\cong S^{-1}R_\sfT.\]
	By assumption the comparison map at $U(s_i)$ is an isomorphism, and so
	\[\Spc(\sfT^c(U(s_i)))\cong \Spec(S^{-1}R_\sfT).\]
	A basis for $\Spec(S^{-1}R_\sfT)$ is given by the collection
	\[\{D (r/s^n_i)\ \vert \ r \in R_\sfT, n \in \ZZ\}.\]
	Using the identification $D(s)\cong \Spec(S^{-1}R_\sfT)$, one can see that for a distinguished basic open $D(r/s^n_i)$ we have
	\[D(r/s^n_i)=D(r)\cap D(s_i) = D(rs_i).\]

	As the comparison map is an isomorphism over $\sfT^c(U(s_i))$ we can take the preimage under $\rho_{U(s_i)}$ to obtain a basis
	\[\{U(rs_i) \ \vert \ r \in R_\sfT\}\]
	for $\Spc(\sfT^c(U(s_i)))$.
	We consider the collection
	\[\mcB=\{U(rs_i) \ \vert \ r \in R_\sfT, s_i \in \mcS\}.\]
	Unwinding the definition of a basis, one can then observe that $\mcB$ is a basis for a topology on $\Spc(\sfT^c)$ and that it coincides with the standard topology generated by the usual basis $\{U(x)\ \vert \ x \in \sfT^c\}$.

\end{proof}

\begin{cor}\label{quasiaffinecomparison}
	If $\sfT$ is quasi-affine then the natural comparison map is injective:
	
	\[\Spc(\sfT^c) \xhookrightarrow{\rho} \Spec(R_\sfT).\]
	
	If $\sfT$ is quasi-affine and the unit $\unit$ satisfies $\Hom(\unit, \Sigma^i \unit)=0$ for all $i >0$, then the natural comparison map is an isomorphism and $\sfT$ is affine.
\end{cor}

\begin{proof}
	By the previous proposition, if $\sfT$ is quasi-affine then the collection $\{U(\cone(s)) \ \vert \ s \in R_\sfT\}$ is a basis for the usual topology on $\Spc(\sfT^c)$. Therefore by \cite[Proposition 3.11]{DSgraded2}, the comparison map is injective.
	For the second part, assume that $\Hom(\unit, \Sigma^i \unit)=0$. Then by \cite[Theorem 7.13]{BaSSS} the comparison map is surjective. Combining this with the quasi-affine assumption we conclude that the comparison map is an isomorphism.
	
\end{proof}

We can use affine categories to give concrete examples of the bad interactions between associated sheaf functors and invertibility of objects.

\begin{defn}
	An \emph{invertible} sheaf on a locally ringed space $(X,\mcO_X)$ is a sheaf $\mcF$ which is locally free of rank $1$. That is, for each point $x \in X$ there is an open neighbourhood $U$ of $x$ such that $\mcF\vert_U \cong \mcO_U.$ 
\end{defn}

First we give an example of an object which is invertible but its associated sheaf is not.

\begin{ex}
	Let $R$ be a commutative ring and consider the category $\mcD^{\perf}(R).$ Define
	\begin{align*}
		X&=\cdots \to 0 \to 0 \to 0 \to 0  \to 0 \to R \to 0 \to \cdots,\\
		Y&= \cdots \to 0 \to R \to 0 \to 0 \to 0 \to 0 \to 0 \to \cdots,
	\end{align*}
	with $X$ concentrated in degree $2$ and $Y$ concentrated in degree $-2$. In $\mcD^{\perf}(R)$ the tensor unit is the complex with $R$ concentrated in degree $0$. Clearly $X\otimes Y \cong \unit.$ As $\mcD^{\perf}(R)$ is affine we have
	\begin{align*}
		[\unit,X]^\# &\cong \widetilde{H^0(X)}\cong 0,\\
		[\unit,Y]^\#&\cong \widetilde{H^0(Y)}\cong 0.
	\end{align*}
	Hence the associated sheaves are not invertible.
\end{ex}

Now we give an example where $X$ is not invertible but $[\unit,X]^\#$ is invertible.

\begin{ex}
	We again consider $\mcD^{\perf}(R)$ and define
	\[X=\cdots \to 0 \to R \to 0 \to R \to 0 \to \cdots,\]
	concentrated in degrees $-2$ and $0$. We have $[\unit, X]^\# \cong \tilde{H^0(X)} \cong \mcO_{\mcD^{\perf}(R)}.$ Therefore $[\unit, X]^\#$ is invertible. However, $X$ is not invertible in $\mcD^{\perf}(R).$
\end{ex}

\section{Mayer-Vietoris covers}
There is a process, defined in \cite{bfgluing}, which allows gluing of objects inside a given triangulated category. We show that gluing behaves well with the associated sheaf functor $[\unit, -]^\#$ with respect to supports and quasi-coherence.
Throughout we consider $\sfT$ a tt-category acting on itself. 

\begin{defn}\cite[2.1]{bfgluing}
	Let $\sfT$ be a triangulated category. A \emph{formal Mayer-Vietoris cover} of $\sfT$ is a pair of thick triangulated subcategories $\mcS_1$ and $\mcS_2$ such that $\Hom_\sfT(X_1,X_2)=\Hom_\sfT(X_2,X_1)=0$ for every pair of objects $X_1\in\mcS_1$ and $X_2 \in \mcS_2.$ We denote by $\mcS_1\oplus \mcS_2$ the thick subcategory whose objects are of the form $X=X_1\oplus X_2$ with $X_1 \in \mcS_1$ and $X_2 \in \mcS_2.$ This setup is referred to as a \emph{Mayer-Vietoris situation.}
\end{defn}

\begin{rem}
	Note that the defintion of a Mayer-Vietoris cover makes no mention of a tensor structure on $\sfT$. Further applications of these covers can be found in \cite{rouquier08}.
\end{rem}
Such a situation can be presented in a square diagram:
\[
\begin{tikzcd}
	\sfT \arrow[r] \arrow[d] & \sfT/\mcS_1 \arrow[d] \\
	\sfT/\mcS_2 \arrow[r]           & \sfT/(\mcS_1\oplus\mcS_2)          
\end{tikzcd}
\]
The content of the next theorem was first in \cite[4.3]{bfgluing} but we present the form given in \cite[2.11]{BaRickard}

\begin{thm}
	Suppose we have a formal Mayer-Vietoris situation as in the above definition. Assume that the quotients in the associated square diagram have small hom objects. Let $X_1 \in \sfT/\mcS_1$ and $X_2 \in \sfT/\mcS_2$ be two objects and $\sigma: X_1 \xrightarrow{\sim}X_2$ an isomorphism in $\sfT/(\mcS_1 \oplus \mcS_2).$ Then there exists an object $X \in \sfT$ and isomorphisms $X\cong X_i$ in $\sfT/\mcS_i$ for $i=1,2,$ compatible with $\sigma$ in $\sfT/(\mcS_1\oplus \mcS_2).$ The object $X$ is unique up to possibly non-unique isomorphism and is called a \emph{gluing} of $X_1$ and $X_2$ along $\sigma.$
\end{thm} 
We may denote a gluing as in the theorem by $X=X_1 \cup_\sigma X_2.$ One may also glue morphisms together under the same hypotheses \cite[3.5]{bfgluing}.

The gluing of objects can be used in the context of tt-categories:

\begin{cor}\cite[5.15]{BaRickard}
	If $\Spc(\sfT^c)=U_1 \cup U_2$ with each $U_i$ a quasi-compact open, then 
	\[
	\begin{tikzcd}
		\sfT \arrow[r] \arrow[d] & \sfT(U_1) \arrow[d] \\
		\sfT(U_2) \arrow[r]           & \sfT(U_1 \cap U_2)         
	\end{tikzcd}
	\]
	satisfies gluing of objects and morphisms as above.
\end{cor}

In general, gluings may not be unique, or are only unique up to non-unique isomorphism. When considering more general covers gluings may not even exist. There are conditions which can be imposed to guarantee existence and uniqueness over finite covers, such as the following theorem.

\begin{thm}\cite[5.13]{bfgluing} 
	\label{connectivegluingthm}
	Let $\Spc(\sfT^c)= U_1 \cup U_2 \cup \cdots \cup U_n$ be a cover by quasi-compact open subsets for $n\geq 2.$ Consider objects $X_i \in \sfT(U_i)$ and isomorphisms $\sigma_{ji} : X_i \xrightarrow{\sim} X_j$ in $\sfT(U_i \cap U_j)$ satisfying the cocycle condition $\sigma_{kj}\sigma_{ji}=\sigma_{ki}$ in $\sfT(U_i\cap U_j \cap U_k)$ for $1\leq i,j,k\leq n.$ Assume moreover the following connectivity condition: for any $i=2,\dots,n$ and for any quasi-compact open $V\subseteq U_i$, we suppose that
	\[
	\Hom_{\sfT(V)}(\Sigma X_i,X_i)=0.
	\]
	Then there exists a gluing which is unique up to unique isomorphism. 
\end{thm}  

\begin{rem}
	As noted in \cite{bfgluing} it suffices to check the connectivity condition only on those $V\subseteq U_i$ which are unions of intersections of $U_1,\dots,U_n.$
\end{rem}

\begin{defn}
	We say a collection of objects admits a \emph{connective gluing} if it satisfies the hypotheses of Theorem \ref{connectivegluingthm}.
\end{defn}

The next two results show that when considering a gluing, the local components contain the expected local information. 

\begin{lem}
	Let $A, X$ be objects in $\sfT$ with $A$ compact. Let $U$ be a quasi-compact open subset of $\Spc(\sfT^c)$. Then for a quasi-compact open subset $V \subseteq U$ we have
	\[\prescript{}{\mathrm{p}}{[}A_U,X_U](V) \cong \prescript{}{\mathrm{p}}{[}A,X](V).\]
\end{lem}

\begin{proof}
	We have the following string of isomorphisms:
	\begin{align*}
		\prescript{}{\mathrm{p}}{[}A_U,X_U](V)&\cong \Hom^{\bullet}_{\sfT(V)}((A_U)_V,(X_U)_V),\\
		&\cong \Hom^{\bullet}_{\sfT(V)}(A_V,X_V),\\
		&\cong \prescript{}{\mathrm{p}}{[}A, X](V),\\
		&\cong\prescript{}{\mathrm{p}}{[}A,X]{\vert_U}(V).
	\end{align*}
\end{proof}

\begin{prop}
	Let $\{U_i\}_{i \in I}$ be a collection of quasi-compact open subsets of $\Spc(\sfT^c)$ such that $\Spc(\sfT^c)=\bigcup_{i \in I}U_i$, for some index set $I$. Suppose we have a collection of objects $X_i \in \sfT(U_i)$ which admit some gluing $X$. Fix a compact object $A$ and define $A_i=A_{U_i}$. If $V$ is a quasi-compact open subset of $\Spc(\sfT^c)$ then 
	\[\prescript{}{\mathrm{p}}{[}A,X](V) \cong \prescript{}{\mathrm{p}}{[}A_i, X_i](V)\]
	for any $i$ such that $V \subseteq U_i$.
\end{prop}

\begin{proof}
	By the definition of the gluing we have $X_{U_i} \cong X_i.$ Suppose $V \subseteq U_i$. Then
	\begin{align*}
		\prescript{}{\mathrm{p}}{[}A,X](V) &\cong \Hom^{\bullet}_{\sfT(V)}(A_V, X_V),\\
		&\cong \Hom^{\bullet}_{\sfT(V)}((A_{U_i})_V, (X_{U_i})_V),\text{ as }V \subseteq U_i,\\
		&\cong \Hom^{\bullet}_{\sfT(V)}((A_{i})_V, (X_{i})_V),\text{ as }X_{U_i} \cong X_i,\\
		&\cong\prescript{}{\mathrm{p}}{[}A_i,X_i](V). 
	\end{align*}
	Finally one can choose any $i$ such that $V \subseteq U_i$ because the gluing data guarantees that if $V \subseteq U_i \cap U_j$ we have 
	\[(X_i)_V \cong (X_{ij})_V \cong (X_j)_V,\] and so any such $i$ will produce the same result. 
\end{proof}

Gluing interacts well with quasi-coherence. 

\begin{prop}
	Let $\Spc(\sfT^c)=U_1 \cup U_2$ with $U_i$ quasi-compact open, $i=1,2,$ and let $X_1, X_2$ be objects of $\sfT$ with an isomorphism $\sigma:X_1 \xrightarrow{\sim} X_2$ in $\sfT(U_1 \cap U_2).$ Let $X=X_1 \cup_\sigma X_2.$
	If $[\unit_1,X_1]^\#$ and $[\unit_2, X_2]^\#$ are both quasi-coherent, then $[\unit, X]^\#$ is quasi-coherent.
\end{prop} 

\begin{proof}
	Fix $\mcP \in \Spc(\sfT^c).$ Without loss of generality suppose $\mcP \in U_1$. As $[\unit_1,X_1]^\#$ is quasi-coherent there exists an open neighbourhood $U$ of $\mcP$ such that there is an exact sequence
	\[
	\mcO^{I}_{\sfT}\vert_U \to \mcO^{J}_{\sfT}\vert_U \to [\unit_1, X_1]^\#\vert_U\to 0. 
	\]
	Moreover we can shrink $U$ to be contained in $U_1$ and quasi-compact. Now on every quasi-compact open subset $V$ of $U_1$ we have $[\unit,X]^\#(V) \cong [\unit_1, X_1]^\#(V)$ and so the given exact sequence for $X_1$ also holds for $X$. Similarly one can obtain such a sequence using $X_2$ in the case where $\mcP \in U_2.$ Therefore $[\unit, X]^\#$ is quasi-coherent.
\end{proof}

\begin{prop}
	Let $\Spc(\sfT^c)=U_1 \cup U_2$ and consider some gluing $ X=X_1\cup_\sigma X_2$ with $X_1 \in \sfT^c(U_1)$ and $X_2 \in \sfT^c(U_2).$ Then
	\[ \supp^\bullet(X,A) \subseteq \supp(X) \subseteq  \supp(X_1 \oplus X_2).\]
\end{prop}

\begin{proof}
	The first containment is Lemma \ref{sheafintt}. 
	We prove the second containment. Consider a prime $\mcP \in \supp(X)$, so that we have $X \not\in \mcP$ and so $X_\mcP \neq 0$. As $\Spc(\sfT^c)=U_1 \cup U_2$ is a cover, we may assume that $\mcP \in U_1$. By Lemma \ref{techstuff} we have an equality $\mcP = \sfT^c_{\supp(\mcP)} = \sfT^c_{\bigcup_{U \ni \mcP}\Spc(\sfT^c) \setminus U}$. Therefore $\sfT^c_{\Spc(\sfT^c)\setminus U_1} \subseteq \mcP$. Therefore 
	\[(X_1)_\mcP \cong (X_{U_1})_\mcP \cong X_\mcP \neq 0.\]
	Therefore $X_1 \not\in \mcP$ and $\mcP \in \supp(X_1)$. We can repeat the argument for each prime, choosing $U_1$ or $U_2$ as appropriate. Therefore 
	\[\supp(X) \subseteq \supp(X_1) \cup \supp(X_2) = \supp(X_1 \oplus X_2).\]
\end{proof}

\begin{rem}
	Providing the gluing exists, the above propositions can be extended to finite collections of objects.
\end{rem}

The next collection of results show that gluings can be used to determine when the presheaf functors are the same as the associated sheaf functors.

\begin{defn}
	Let $U$ be a quasi-compact open subset of $\Spc(\sfT^c)$ and let $\mcU=\{U_1,\dots,U_n\}$ be an open cover of $U$ by quasi-compact opens. We cay an object $X \in \sfT$ is \emph{1-coconnected over} $\mcU$ if for every $i$ and every quasi-compact open $V \subseteq U_i$ we have
	\[
	\Hom^\bullet_{\sfT(V)}(\unit, \Sigma^{-1}X)=0,
	\]
	We will say that $X$ is \emph{absolutely 1-coconnected over} $U$ if it is 1-coconnected over every finite cover of $U$. 
\end{defn}

\begin{thm}
	\label{sheafcondition}
	Let $X$ be an object of $\sfT$, $U$ a quasi-compact open subset of $\Spc(\sfT^c)$ and $\mcU=\{U_1,\dots,U_n\}$ a cover of $U$ by quasi-compact opens. If $X$ is 1-coconnected over $\mcU$ then $\prescript{}{\mathrm{p}}{[}\unit, X]$ satisfies the sheaf condition with respect to this cover. That is, the sequence
	\[\prescript{}{\mathrm{p}}{[}\unit,X](U)\to \bigoplus_i \prescript{}{\mathrm{p}}{[}\unit,X](U_i)\to \bigoplus_{i,j}\prescript{}{\mathrm{p}}{[}\unit,X](U_{ij})\]
	is exact, where $U_{ij}=U_i \cap U_j.$ 
\end{thm}

\begin{proof}
	Consider the base case $n=2$, where the cover $\mcU$ has two opens $U_1$ and $U_2.$ For each $i \in \ZZ$ the Mayer-Vietoris sequence of \cite[Theorem 3.5]{bfgluing} gives an exact sequence
	\[\begin{tikzpicture}[scale=2]
		\matrix(m)[matrix of math nodes,column sep=15pt,row sep=15pt]{
			\Hom_{\sfT(U_{12})}(\Sigma \unit,u^{\otimes i}\otimes X) &\Hom_{\sfT(U)}(\unit,u^{\otimes i}\otimes X) \\
			\Hom_{\sfT(U_1)}(\unit,u^{\otimes i}\otimes X)\oplus\Hom_{\sfT(U_2)}(\unit,u^{\otimes i}\otimes X) &\Hom_{\sfT(U_{12})}(\unit,u^{\otimes i}\otimes X). \\
		};
		\draw[->,font=\scriptsize,every node/.style={above},rounded corners]
		(m-1-1) edge (m-1-2) 
		(m-1-2.east) --+(5pt,0)|-+(0,-7.5pt)-|([xshift=-5pt]m-2-1.west)--(m-2-1.west)
		(m-2-1) edge (m-2-2)
		;
	\end{tikzpicture}.\]
	By assumption $X$ is 1-coconnected over $\mcU$ so for all $i$, $\Hom_{\sfT(U_{12})}(\Sigma\unit,u^{\otimes i}\otimes X)=0.$ For each $i$ the sequence becomes
	
	\[\begin{tikzpicture}[scale=2]
		\matrix(m)[matrix of math nodes,column sep=15pt,row sep=15pt]{
			0 &\Hom_{\sfT(U)}(\unit,u^{\otimes i}\otimes X) \\
			\Hom_{\sfT(U_1)}(\unit,u^{\otimes i}\otimes X)\oplus\Hom_{\sfT(U_2)}(\unit,u^{\otimes i}\otimes X) &\Hom_{\sfT(U_{12})}(\unit,u^{\otimes i}\otimes X). \\
		};
		\draw[->,font=\scriptsize,every node/.style={above},rounded corners]
		(m-1-1) edge (m-1-2) 
		(m-1-2.east) --+(5pt,0)|-+(0,-7.5pt)-|([xshift=-5pt]m-2-1.west)--(m-2-1.west)
		(m-2-1) edge (m-2-2)
		;
	\end{tikzpicture}.\]

	Therefore on the presheaf we have an exact sequence 
	\[0\to \prescript{}{\mathrm{p}}{[}\unit, X](U)\to\prescript{}{\mathrm{p}}{[}\unit, X](U_1)\oplus\prescript{}{\mathrm{p}}{[}\unit, X](U_2)\to\prescript{}{\mathrm{p}}{[}\unit, X](U_{12}).\]
	
	Therefore the presheaf condition holds on the cover $\mcU$. Now suppose the result holds for all cases $n<m$ with $m>2.$ Consider the case $n=m$. Given a cover $\mcU=\{U_i\}^m_{i=1}$ let $V=\bigcup^n_{i=2}U_i$ and consider the cover $U=U_1 \cup V$. By induction the sheaf condition holds on this cover of $U$ and the given cover of $V$. Therefore the condition holds for the original cover $\mcU$ and the result holds. 
\end{proof}

\begin{cor}
	Let $U$ be a quasi-compact open subset of $\Spc(\sfT^c)$ and suppose $X \in \sfT$ is absolutely 1-coconnected over $U$. Then $\prescript{}{\mathrm{p}}{[}\unit, X]$ verifies the sheaf condition at $U$ and hence
	\[[\unit,X]^\#(U)=\prescript{}{\mathrm{p}}{[}\unit, X](U)=\Hom^\bullet_{\sfT(U)}(\unit,X).
	\]
\end{cor}

\begin{proof}
	Recall that the sheaf condition is that for every finite cover $\mcU=\{U_1,\dots, U_n\}$ of $U$ the sequence
	\[\prescript{}{\mathrm{p}}{[}\unit,X](U)\to \bigoplus_i \prescript{}{\mathrm{p}}{[}\unit,X](U_i)\to \bigoplus_{i,j}\prescript{}{\mathrm{p}}{[}\unit,X](U_{ij})\]
	is exact, where $U_{ij}=U_i \cap U_j.$ As $X$ is assumed to be absolutely 1-coconnected over $U$, it is 1-coconnected with respect to each finite cover $\mcU$. Therefore by Theorem \ref{sheafcondition} the given sequence is exact for every finite cover of $U$, and therefore the sheaf condition is satisfied. Therefore we have an equality of sections
	\[[\unit,X]^\#(U)=\prescript{}{\mathrm{p}}{[}\unit, X](U)=\Hom^\bullet_{\sfT(U)}(\unit,X)\]
	as required. 
	
\end{proof}

We have seen that a gluing of objects can lead to good properties in the associated sheaves. However, a gluing of the associated sheaves does not give a gluing of objects even in the simplest cases.

\begin{ex}
	Consider $\sfT^c=\mcD^{\perf}(R)$ for some commutative ring $R$. This is an affine category and moreover for each object $X$ we have $[\unit,X]^\# \cong \widetilde{H^0(X)}$. Consider the following two perfect complexes:
	\begin{align*}
		X&=\cdots \to 0\to R \to 0 \to R \to 0 \to 0 \to \cdots,\\
		Y&=\cdots \to 0 \to 0 \to 0 \to R \to 0 \to 0 \to \dots,
	\end{align*}
	where $X$ has $R$ in degrees $-2$ and $0$, and $Y$ has $R$ in degree $0$ only. Both complexes have the same zeroth cohomology and so there is an isomorphism of untwisted sheaves $[\unit,X]^\# \cong [\unit, Y]^\#.$ However they are not isomorphic in $\mcD^{\perf}(R)$ as they have non-isomorphic homology in degree $-2$. Therefore consider the trivial cover $U_1=U_2=\Spc(\sfT^c)$. We have
	\[\sfT^c(U_1 \cap U_2)\simeq \sfT^c(\Spc(\sfT^c))\simeq \sfT^c,\]
	and the existence of a gluing of $X$ to $Y$ would imply that $X\cong Y$ in $\mcD^{\perf}(R)$, a contradiction. Therefore the existence of a gluing of sheaves does not provide a gluing objects.
\end{ex}

\section{Interactions with t-structures }
We now consider the case where $\sfT$ is equipped with a t-structure. For more details on t-structures see \cite{fperv}.

\begin{defn}
	Let $\sfT$ be a triangulated category with full subcategories $\sfT^{\leq 0}$ and $\sfT^{\geq 0 }.$ Denote $\sfT^{\leq n}=\Sigma^{-n}\sfT^{\leq 0}$ and $\sfT^{\geq n}=\Sigma^{-n}\sfT^{\geq 0}.$ A \emph{t-structure} on $\sfT$ is the data $(\sfT^{\leq 0},\sfT^{\geq 0})$ such that
	\begin{enumerate}
		\item There are containments 
		\[\sfT^{\leq -1}\subseteq \sfT^{\leq 0}\text{ and }\sfT^{\geq 1}\subseteq \sfT^{\geq 0}.\]
		\item For any $X \in \sfT^{\leq 0},Y \in \sfT^{\geq 1}$, $\Hom_\sfT(X,Y)=0.$
		\item For any $X \in \sfT$ there exists a distinguished triangle 
		\[X_0 \to X \to X_1 \to \Sigma X_0\]
		such that $X_0 \in \sfT^{\leq 0}$ and $X_1 \in \sfT^{\geq 1}.$ 
	\end{enumerate}
	A t-structure is \emph{non-degenerate} if
	\[
	\bigcap_{n\in\ZZ}\sfT^{\leq n}=\bigcap_{n\in\ZZ}\sfT^{\geq n}=0.
	\]
	The \emph{heart} of the t-structure is the full subcategory 
	\[\sfT^\heartsuit=\sfT^{\leq 0}\cap \sfT^{\geq 0}.\]
\end{defn} 
The following are standard results on t-structures:
\begin{prop}
	Let $(\sfT^{\leq 0},\sfT^{\geq 0})$ be a t-structure on $\sfT$. Then the following hold:
	\begin{enumerate}
		\item The heart $\sfT^\heartsuit$ is an abelian category.
		\item Consider the natural inclusion functors $i^{\leq n}:\sfT^{\leq n} \to \sfT$ and $i^{\geq n}:\sfT^{\geq n} \to \sfT.$ There exist functors $\tau^{\leq n}:\sfT \to \sfT^{\leq n}$ and $\tau^{\geq n}:\sfT \to \sfT^{\geq n}$ such that
		\[i^{\leq n}\dashv \tau^{\leq n}\text{ and }\tau^{\geq n} \dashv i^{\geq n}. \]
	\end{enumerate}
	Explicitly 
	\begin{align*}
		\Hom_\sfT(i^{\leq n}X,Y) &\cong \Hom_{\sfT^{\leq n}}(X,\tau^{\leq n}Y),\\
		\Hom_\sfT(X,i^{\geq n}Y)&\cong \Hom_{\sfT^{\geq n}}(\tau^{\geq n}X,Y).
	\end{align*}
\end{prop}

\begin{defn}
	Given a t-structure $(\sfT^{\leq 0},\sfT^{\geq 0})$ we define the $n$-th cohomology functor $H^n:\sfT\to \sfT^\heartsuit$ by\[H^n(X)=\tau^{\geq 0}\tau^{\leq 0}\Sigma^n X.\]
\end{defn}

We give the following theorem from \cite{compactstruc} showing that a t-structure can be built from a compact object under suitable conditions.

\begin{defn}
	We say an object $X \in \sfT$ is \emph{connective} if for all $n>0$, $\Hom_\sfT(X, \Sigma^n X)=0$. If $U \subseteq \Spc(\sfT^c)$ is a quasi-compact open subset then we say that an object $X$ is \emph{locally connective over }$U$ if $\Hom_{\sfT(U)}(X,\Sigma^n X)=0$ for all $n>0$. 
	If $\mcU=(U_i)_{i \in I}$ is an open cover of $\Spc(\sfT^c)$ with each $U_i$ quasi-compact, then we say that $X$ is \emph{locally connective over }$\mcU$ if $X$ is locally connective over each of the $U_i$.
\end{defn}

\begin{thm}\cite[1.3]{compactstruc}
	\label{hkmtstruc}
	Let $\sfT$ be a triangulated category with arbitrary coproducts, $C$ a connective compact object, and $B=\End_\sfT(C)^\op.$ Define 
	\begin{align*}
		\sfT^{\leq n}&=\{X \in \sfT \ \vert \ \Hom_\sfT(C,\Sigma^i X)=0 \text{ for }i>n\}\\
		\sfT^{\geq n}&=\{X \in \sfT \ \vert \ \Hom_\sfT(C,\Sigma^i X)=0 \text{ for }i<n\}
	\end{align*}
	Let $\sfT^\heartsuit=\sfT^{\leq 0}\cap \sfT^{\geq 0}.$
	If $\{\Sigma^i C \ \vert \ i \in \ZZ\}$ is a generating set, then the following hold:
	\begin{enumerate}
		\item $(\sfT^{\leq 0}, \sfT^{\geq 0})$ is a non-degenerate t-structure on $\sfT$.
		\item The functor
		\[\Hom_{\sfT}(C,-):\sfT^\heartsuit \to \Modu B,\]
		is an equivalence of categories.
	\end{enumerate}
\end{thm}

\begin{defn}
	Let $C$ be a connective compact object of $\sfT$. Let $(\sfT^{\leq 0}, \sfT^{\geq 0})$ be the t-structure of Theorem \ref{hkmtstruc}. We say that $(\sfT^{\leq 0}, \sfT^{\geq 0})$ is the t-structure \emph{connectively generated by }$C$. We say that $C$ \emph{connectively generates} $(\sfT^{\leq 0}, \sfT^{\geq 0})$. 
\end{defn}

Observe that if $\sfT$ is generated by the tensor unit $\unit$ and $\unit$ is locally connective with respect to some cover $\mcU$, then for each $U\in\mcU$ the tt-category $\sfT(U)$ can be equipped with the t-structure connectively generated by $\unit$. We denote the heart of this structure by $\sfT(U)^\heartsuit$.

We now detail conditions under which these local t-structures interact well with the corresponding localisation functors.

\begin{prop}
	\label{localtrunc}
	Let $\sfT$ be a big tt-category, generated by the tensor unit $\unit$. Consider a morphism $s \in R_{\sfT} \coloneqq \Hom_{\sfT}(\unit, \unit)$ and the quasi-compact open subset $U=U(\cone(s))$.
	\begin{enumerate}
		\item For an object $X \in \sfT$, if $\Hom_\sfT(\unit, X) \cong 0$ then $\Hom_{\sfT(U)}(\unit_U, X_U)=0$. \\
		\item If $\unit$ is connective in $\sfT$ then $\unit_U$ is connective in $\sfT(U)$. As a consequence, if $\unit$ connectively generates a t-structure on $\sfT$ then $\unit_U$ connectively generates a t-structure on $\sfT(U)$.  \\
		\item If $\unit$ is connective then for all $i \in \ZZ$ we have the following commutative diagrams:
		
		\[
		\begin{tikzcd}
			\sfT \arrow[rr, "(-)_U"] \arrow[dd, "\tau^{\leq i}"'] &  & \sfT(U) \arrow[dd, "\tau^{\leq i}_{U}"] \\
			&  &                     \\
			\sfT^{\leq i} \arrow[rr, "(-)_U"]                  &  & \sfT(U)^{\leq i}                  
		\end{tikzcd}
		\text{ and }
		\begin{tikzcd}
			\sfT \arrow[rr, "(-)_U"] \arrow[dd, "\tau^{\geq i}"'] &  & \sfT(U) \arrow[dd, "\tau^{\geq i}_{U}"] \\
			&  &                     \\
			\sfT^{\geq i} \arrow[rr, "(-)_U"]                  &  & \sfT(U)^{\geq i}                  
		\end{tikzcd}
		\] 
		In other words, truncation and localisation commute.
	\end{enumerate}
\end{prop}

\begin{proof}
	\begin{enumerate}
		\item By Theorem \ref{hpstheorem} we have an isomorphism
		\[\Hom_{\sfT(U)}(\unit_U, X_U)\cong S^{-1}\Hom_\sfT(\unit, X)\]
		where the right term is the localisation of the module $\Hom_\sfT(\unit, X)$ with respect to the multiplicative set associated to $s$. By assumption $\Hom_\sfT(\unit, X)$ is zero, and therefore so is the localisation. Therefore $\Hom_{\sfT(U)}(\unit, X)\cong 0$.\\
		
		\item The unit $\unit$ is connective and so $\Hom_\sfT(\unit, \Sigma^{n}\unit)\cong 0$ for all $n>0$. Applying part $(1)$ to each of the $\Sigma^n \unit$ we have that for each $n>0$, $\Hom_{\sfT(U)}(\unit, \Sigma^n \unit) \cong 0$. In other words, $\unit$ is locally connective over $U$ and therefore connectively generates a t-structure on $\sfT(U)$.\\
		
		\item First note that by part (2), as $\unit$ is connective it makes sense to consider the t-structure $(\sfT(U)^{\leq 0},\sfT(U)^{\geq 0})$ connectively generated by $\unit_U$. Fix $X \in \sfT$ and $i \in \ZZ$. We first prove that
		\[
		(\tau^{\leq i}X)_U \in \sfT(U)^{\leq i}\text{ and }(\tau^{\geq i}X)_U \in \sfT(U)^{\geq i}.
		\]
		Consider $\Hom_{\sfT(U)}(\unit_U, \Sigma^j (\tau^{\leq i}X)_U).$ As $\tau^{\leq i}X \in \sfT^{\leq i}$ we have $\Hom_{\sfT}(\unit, \Sigma^j \tau^{\leq i}X)=0$ for all $j>i$. Therefore by (1) $\Hom_{\sfT(U)}(\unit_U, \Sigma^j (\tau^{\leq i}X)_U)=0$ for all $j>i$ so by definition $(\tau^{\leq i}X)_U \in \sfT(U)^{\leq i}$. A similar proof provides the second part of the claim.
		Now consider the diagrams
		\[
		\begin{tikzcd}
			\sfT \arrow[rr, "(-)_U"] \arrow[dd, "\tau^{\leq i}"'] &  & \sfT(U) \arrow[dd, "\tau^{\leq i}_{U}"] \\
			&  &                     \\
			\sfT^{\leq i} \arrow[rr, "(-)_U"]                  &  & \sfT(U)^{\leq i}                  
		\end{tikzcd}
		\text{ and }
		\begin{tikzcd}
			\sfT \arrow[rr, "(-)_U"] \arrow[dd, "\tau^{\geq i}"'] &  & \sfT(U) \arrow[dd, "\tau^{\geq i}_{U}"] \\
			&  &                     \\
			\sfT^{\geq i} \arrow[rr, "(-)_U"]                  &  & \sfT(U)^{\geq i}                  
		\end{tikzcd}
		\]
		The proof of the first claim shows that the bottom arrows in each diagram do in fact land in $\sfT(U)^{\leq i}$ and $\sfT(U)^{\geq i}$ respectively.
		
		Recall that the t-structure on $\sfT$ induces a unique triangle 
		\[\tau^{\leq i}X \to X \to \tau^{\geq i+1}X \to \Sigma \tau^{\leq i}X.\] Applying the localisation funtor $(-)_U$ gives a triangle
		\[(\tau^{\leq i}X)_U \to X_U \to (\tau^{\geq i+1}X)_U \to (\Sigma \tau^{\leq i}X)_U.\]
		Now by the previous lemma $(\tau^{\leq i}X)_U \in \sfT(U)^{\leq i}$ and $(\tau^{\geq i+1}X)_U \in \sfT(U)^{\geq i+1}$. Therefore the above triangle is isomorphic to the canonical triangle induced by the t-structure on $\sfT(U)$, given by the following diagram:
		\[
		\begin{tikzcd}
			(\tau^{\leq i}X)_U \arrow[r] \arrow[d] & X_U \arrow[r] \arrow[d] & (\tau^{\geq i+1}X)_U \arrow[r] \arrow[d] & (\Sigma \tau^{\leq i}X)_U \arrow[d] \\
			\tau^{\leq i}_U(X_U) \arrow[r]           & X_U \arrow[r]           & \tau^{\geq i+1}_U(X_U) \arrow[r]           & \Sigma \tau^{\leq i}_U(X_U)          
		\end{tikzcd}\]
		As all of the vertical arrows are isomorphisms we have 
		\[
		(\tau^{\leq i}X)_U \cong \tau^{\leq i}_U(X_U) \text{ and } (\tau^{\geq i+1}X)_U \cong \tau^{\geq i+1}_U(X_U)
		\] 
		which completes the proof. 
	\end{enumerate}
\end{proof}

\begin{rem}
	\label{loccoh}
	As an immediate consequence we have
	\[
	H^0_U(X_U) \cong (H^0(X))_U
	\]
	as the homology functor is defined by the truncation functors. 
\end{rem}

Our work on affine categories allows us to define this t-structure in terms of sheaves associated to objects. We work in the untwisted setting. 

\begin{lem}
	Let $\sfT$ be an affine category and suppose that the unit $\unit$ connectively generates a t-structure on $\sfT$. Then:
	\begin{enumerate}
		\item For all $n \in \ZZ$:
		\begin{enumerate}
			\item $\sfT^{\leq n} = \{X \in \sfT\ \vert \ [\unit, \Sigma^i X]^\#=0\ i>n\}.$
			\item $\sfT^{\geq n} = \{X \in \sfT\ \vert \ [\unit, \Sigma^i X]^\#=0\ i<n\}.$
			\item $\sfT^{\heartsuit} = \{X \in \sfT\ \vert \ [\unit, \Sigma^i X]^\#=0\ i\neq 0\}.$
		\end{enumerate} 
		\item Let $X \in \sfT^c.$ Then
		\begin{enumerate}
			\item If $X \in \sfT^{\leq n}$ then $\supp X=\bigcup_{i\leq n}\supp(\Sigma^i X, \unit).$  
			\item If $X \in \sfT^{\geq n}$ then $\supp X=\bigcup_{i \geq n}\supp(\Sigma^i X, \unit).$
			\item If $X \in \sfT^\heartsuit$ then $\supp X = \supp(X, \unit).$
		\end{enumerate} 
	\end{enumerate}
\end{lem}

\begin{proof}
	As $\sfT$ is affine the sheaf $[\unit, \Sigma^i X]^\#$ is equivalent to the sheaf associated to the module $\Hom_\sfT(\unit, \Sigma^i X)$. For $(1)$, by Theorem \ref{hkmtstruc} we have $\sfT^{\leq n}=\{X \in \sfT \ \vert \ \Hom_\sfT(\unit, \Sigma^i X)=0 \ i > n \}$. As $\sfT$ is affine, for each $i > n$, $\Hom_\sfT(\unit, \Sigma^i X)=0$ if and only if $[\unit, \Sigma^i X]^\#=0$. Therefore $\sfT^{\leq n} = \{X \in \sfT\ \vert \ [\unit, \Sigma^i X]^\#=0\ i>n\},$ proving (1)(a). Similar proofs give us (1)(b) and (1)(c). For (2) note that by Remark \ref{untwisteddecomp} and Theorem \ref{thicksupcompare} we have for compact $X$ equalities
	\[\supp X = \supp^\bullet(X, \unit)= \bigcup_{i \in \ZZ}\supp(\Sigma^i X, \unit),\]
	where $\supp^\bullet(X, \unit)$ is twisted with respect to $\Sigma \unit$. Now if $X \in \sfT^{\leq n}$, by part (1)(a) $[\unit,\Sigma^i X]^\#=0$ for $i < n$. Therefore by definition $\supp(\Sigma^i X, \unit)=\varnothing$ for all $i < n$ and so $\bigcup_{i \in \ZZ}\supp(\Sigma^i x, \unit)=\bigcup_{i \leq n}\supp(\Sigma^i X, \unit)$, proving (2)(a). Similar proofs provide the results in (2)(b) and (2)(c), completing the proof.
\end{proof}

We want to investigate when an object can be analysed via an appropriate object in the heart of the t-structure, as in the following definition.

\begin{defn}
	Given an object $X \in \sfT$ we say that an object $X^\heartsuit \in \sfT^\heartsuit$ is a \emph{hearty replacement} for $X$ if 
	\[[\unit, X]^\#=[\unit, X^\heartsuit]^\#.\]
	We say that $X$ can be \emph{heartily replaced} by $X^\heartsuit$. 
\end{defn}

When the comparison map of spectra is an isomorphism then the heart of the t-structure encodes information about associated sheaf functors in the obvious way.

\begin{prop}
	\label{firstheart}
	Suppose $\sfT$ is affine and that the unit object connectively generates a t-structure. Then 
	\begin{itemize}
		\item $\sfT^\heartsuit\simeq \QCoh\Spec(R_{\sfT}).$
		\item Each object $X \in \sfT$ has a hearty replacement $X^\heartsuit$. That is, for an object $X \in \sfT$ there exists an object $X^\heartsuit \in \sfT^\heartsuit$ such that 
		\[[\unit, X]^\# \cong [\unit,X^\heartsuit]^\#.\] 
	\end{itemize}
\end{prop}

\begin{proof}
	Consider the following diagram:
	\[
	\begin{tikzcd}
		\sfT^\heartsuit \arrow[rr, "{\Hom_\sfT(\unit,-)}"] \arrow[rrdd, "{[\unit,-]^\#}"'] &  & \Modu R_\sfT \arrow[dd, "\widetilde{(-)}"] \\
		&  &                                                       \\
		&  & \QCoh\Spec(R_\sfT)                              
	\end{tikzcd}
	\]
	The top arrow is an equivalence by Theorem \ref{hkmtstruc}. The vertical arrow is an equivalence as $\Spec(R_{\sfT})$ is an affine scheme. As $\sfT$ is affine, for any object $X \in \sfT$ we have $[\unit,X]^\#\cong \widetilde{\Hom_\sfT(\unit,X)}$ i.e. $[\unit,-]^\# \cong \widetilde{(-)}\circ \Hom_\sfT(\unit,-)$. Therefore our diagram commutes and the final diagonal arrow is an equivalence, as required. For the second part, note that $[\unit,X]^\#$ is quasi-coherent for all $X \in \sfT$ as $\sfT$ is affine, so via the equivalence in our diagram there must exist $X^\heartsuit \in \sfT^\heartsuit$ such that $[\unit, X]^\# \cong [\unit, X^\heartsuit]^\#$ as claimed.
\end{proof}

Suppose instead that $\sfT$ is not affine. Then the diagram of Proposition \ref{firstheart} may not be commutative. Define a functor \[(-)^\flat = \widetilde{\Hom_\sfT(\unit,-)}\] and consider the diagram 

\[
\begin{tikzcd}
	\sfT \arrow[rr, "{\Hom_\sfT(\unit,-)}"] \arrow[rrdd, "(-)^\flat"] \arrow[dd, "{[\unit,-]^\#}"'] &  & \Modu R_{\sfT} \arrow[dd, "\widetilde{(-)}"] \\
	&  &                                                      \\
	\Shv\Spec(\sfT)                                                                                                             &  & \QCoh\Spec(R_{\sfT}) \arrow[ll, "\rho^*"']       
\end{tikzcd}\]
By definition the upper triangle commutes. If the lower triangle commuted then $[\unit, X]^\#$ would be quasi-coherent for all $X \in \sfT$. We show that under mild conditions the lower triangle fails to be commutative.

\begin{prop}
	\label{countershort}
	Suppose that $\Spc(\sfT^c)$ contains a non-zero prime ideal and that \[\Hom_{\sfT^c}(\unit,\unit)=k\] is a field. Then 
	\[[\unit, -]^\# \not\cong \rho^\ast (-)^\flat.\]
\end{prop}

\begin{proof}
	Consider an sheaf $\mcF$ in \(\Modu\mcO_k\). Such an object is in the additive closure of the structure sheaf and so the pull back $\rho^\ast \mcF$ must be in the additive closure of $\mcO_{\sfT}$. Every object in this closure must be supported everywhere. Now consider an object $M \in \mcP$, where $\mcP$ is a non-zero prime ideal in $\Spc(\sfT^c)$. The associated sheaf $[\unit,M]^\#$ is not supported at $\mcP$ and so $[\unit,M]^\#$ is not in the additive closure of $\mcO_{\sfT}$. Therefore $[\unit,M]^\#$ is not the pullback of a $\mcO_k$-module along $\rho$ and so $[\unit, -]^\# \not\cong \rho^\ast (-)^\flat$.
\end{proof}

In the affine case we have the following connections between the cohomology functor and the associated sheaves.

\begin{lem}
	\label{littlecohlemma}
	If $\sfT$ affine then for an object $X$ we have
	\begin{enumerate}
		\item $[\unit,X]^\# \cong [\unit, \tau^{\leq 0}X]^\#.$
		\item $[\unit, H^0X]^\# \cong [\unit, \tau^{\geq 0}X]^\#.$
		\item If $X \in \sfT^{\geq 0}$ then $[\unit, X]^\# \cong [\unit,H^0X]^\#.$
	\end{enumerate}
\end{lem}

\begin{proof}
	\begin{enumerate}
		\item First note that as the unit $\unit$ is connective, $\Hom_\sfT(\unit, \Sigma^i \unit)=0$ for all $i > 0$. In other words, $\unit \in \sfT^{\leq 0}$, which is a full subcategory. Using the adjunction between the associated truncation and inclusion functors for $\sfT^{\leq 0}$ we obtain the following string of isomorphisms:
		\begin{align*}
			\Hom_\sfT(\unit, \tau^{\leq 0}X) &\cong \Hom_{\sfT^{\leq0}}(\unit, \tau^{\leq 0}X),\\
			&\cong \Hom_\sfT(i^{\leq 0}\unit, X),\\
			&\cong \Hom_\sfT(\unit, X).
		\end{align*}
		By assumption $\sfT$ is affine. Combining this fact with the above isomorphism we conclude
		\[[\unit, X]^\# \cong \widetilde{\Hom_\sfT(\unit, X)} \cong \widetilde{\Hom_\sfT(\unit, \tau^{\leq 0}X)} \cong [\unit, \tau^{\leq 0}X]^\#.\]
		\item By definition $H^0 X = \tau^{\leq 0}\tau^{\geq 0}X$. Using part (1) we obtain
		\[\Hom_\sfT(\unit, H^0 X)=\Hom_\sfT(\unit, \tau^{\leq 0}\tau^{\geq 0}X)\cong \Hom_\sfT(\unit, \tau^{\geq 0}X).\]
		Again, as $\sfT$ is affine this gives us $[\unit, H^0X]^\# \cong [\unit, \tau^{\geq 0}X]^\#$ as required.
		\item If $X \in \sfT^{\geq 0}$ then $X \cong \tau^{\geq 0}X$ and so
		\[[\unit, X]^\# \cong [\unit, \tau^{\geq 0}X]^\#  \cong [\unit,H^0X]^\#,\]
		where the second isomorphism is from (2).
	\end{enumerate}
\end{proof}

In our hunt for hearty replacements the previous lemma states that if $X\in \sfT^{\geq 0}$ then the zeroth cohomology $H^0 X$ is a hearty replacement for $X$.  
For $\sfT$ affine we can extend this to all objects in $\sfT$ and upgrade the assignment $(-)^\heartsuit$ to a functor via the cohomology functor.

\begin{prop}
	\label{affinecohom}
	Suppose that $\sfT$ has the t-structure connectively generated by $\unit$. Then for all $X \in \sfT$ we have
	\[\Hom_{\sfT}(\unit,X)\cong \Hom_{\sfT}(\unit,H^0X),\]
	and if $\sfT$ is affine we have
	\[[\unit, X]^\# \cong [\unit, H^0X]^\#,\]
	where $H^0$ is the cohomology functor induced by the t-structure connectively generated by $\unit$. 
	We can choose the functor $H^0(-)$ as our assignment $(-)^\heartsuit$. 
\end{prop}

\begin{proof}
	For $X \in \sfT$ consider the canonical triangle
	\[\tau^{\leq -1}X \to X \to \tau^{\geq 0} X \to \Sigma \tau^{\leq -1}X.\]
	Applying the homological functor $\Hom_\sfT(\unit, -)$ produces a long exact sequence

	\[\begin{tikzpicture}[scale=2]
		\matrix(m)[matrix of math nodes,column sep=15pt,row sep=15pt]{
			\cdots \to \Hom_\sfT(\unit, \tau^{\leq -1}X) & \Hom_\sfT(\unit, X) \\
			\Hom_\sfT(\unit, \tau^{\geq 0}X) &\Hom_\sfT(\unit, \Sigma\tau^{\leq -1}X) \to \cdots \\
		};
		\draw[->,font=\scriptsize,every node/.style={above},rounded corners]
		(m-1-1) edge (m-1-2) 
		(m-1-2.east) --+(5pt,0)|-+(0,-7.5pt)-|([xshift=-5pt]m-2-1.west)--(m-2-1.west)
		(m-2-1) edge (m-2-2)
		;
	\end{tikzpicture}.\]
	As $\tau^{\leq -1}X \in \sfT^{\leq -1}$ we have \[\Hom_\sfT(\unit, \tau^{\leq -1}X) \cong \Hom_\sfT(\unit, \Sigma^0 \tau^{\leq -1}X) \cong 0\] and $\Hom_\sfT(\unit, \Sigma^0 \tau^{\leq -1}X)\cong 0$. Therefore we obtain the short exact sequence
	\[0 \to \Hom_\sfT(\unit, X) \to \Hom_\sfT(\unit, \tau^{\geq 0}X)\to 0\]
	which by exactness gives $\Hom_\sfT(\unit, X) \cong \Hom_\sfT(\unit, \tau^{\geq 0}X)$. As we assumed $\sfT$ affine this gives $[\unit, X]^\# \cong [\unit, \tau^{\geq 0}X]^\#$. Then by Lemma \ref{littlecohlemma} (2) we have 
	\[[\unit, X]^\# \cong [\unit, \tau^{\geq 0}X]^\# \cong [\unit, H^0 X]^\#.\] 
\end{proof}

\begin{rem}
	While Proposition \ref{affinecohom} applies to affine tt-categories, issues arise when trying to extend it to schematic categories. Indeed, Proposition \ref{localtrunc} only guarantees that the cohomology and localisation functors are well behaved when localising over an open of the form $U(\cone(s))$ for some $s \in R_\sfT$. A quasi-affine tt-category $\sfT$ could therefore be a good candidate to extend the proposition. Unfortunately, requiring that the unit $\unit$ connectively generate a t-structure on $\sfT$ means that the hypotheses of Corollary \ref{quasiaffinecomparison} are satisfied, and so our candidate quasi-affine category is in fact affine.   
\end{rem}

\begin{thm}
	Let $\sfT$ be an affine category such that the tensor unit $\unit$ generates $\sfT$ and connectively generates a t-structure on $\sfT$. Then for all objects $X \in \sfT$ we have
	\[\supp^\bullet(X, \unit) = \bigcup_{i \in \ZZ}\supp(H^i(X), \unit),\]
	where the support on the left is twisted by the invertible object $\Sigma \unit$.
\end{thm}

\begin{proof}
	
	By the previous proposition, we identify the associated sheaf of an object to that of its zeroth cohomology. Applying this to the definition of the $\Sigma\unit$-twisted support we obtain the following equalities: 
	\begin{align*}
		\supp^\bullet(X, \unit) &= \bigcup_{i \in \ZZ}\supp(\Sigma^i X, \unit),\text{ by Remark \ref{untwisteddecomp},}\\
		&=\bigcup_{i \in \ZZ}\supp(H^0\Sigma^i X, \unit),\text{ by the previous theorem,}\\
		&=\bigcup_{i \in \ZZ}\supp(H^i (X), \unit),
	\end{align*}
	completing the proof.
\end{proof}

\appendix

\typeout{}
\bibliography{bibbroke}

\end{document}